\documentclass[12pt,reqno]{amsart} 
\usepackage[T1]{fontenc}
\usepackage{amsfonts}
\usepackage{amssymb}
\usepackage{mathrsfs}
\usepackage[latin1]{inputenc}
\usepackage{amsmath}
\usepackage{paralist}
\usepackage[english]{babel}
\usepackage{setspace}
\usepackage{bbm}
\usepackage{fancyhdr}

\newtheorem{theorem}{Theorem}[section]
\newtheorem{lemma}[theorem]{Lemma}

\newtheorem{remark}[theorem]{Remark}
\newtheorem{prop}[theorem]{Proposition}
\newtheorem{example}[theorem]{Example}
\newtheorem{corollary}[theorem]{Corollary}

\numberwithin{equation}{section}
\newcommand{\R}{{\mathbb R}}

\newcommand{\C}{{\mathbb C}}
\newcommand{\N}{{\mathbb N}}

\newcommand{\cL}{{\mathcal L}}

\newcommand{\cE}{{\mathcal E}}

\newcommand{\cB}{{\mathcal B}}

\newcommand{\cS}{{\mathcal S}}
\newcommand{\cU}{{\mathcal U}}

\newcommand{\be}{\beta}

\newcommand{\si}{\sigma}

\newcommand{\ov}{\overline}

\newcommand\Ker{\mathop{\rm Ker}}

\newcommand{à}{\`a}

\begin{document}

\title[Generalized Ces\`aro operators in  BK-sequence spaces]{Mean ergodic and related properties of generalized Ces\`aro operators in  BK-sequence spaces}

\author[A.\,A. Albanese, J. Bonet and W.\,J. Ricker]{Angela\,A. Albanese, Jos\'e Bonet and Werner\,J. Ricker}

\thanks{\textit{Mathematics Subject Classification 2020:}
Primary 46B42, 46B45, 47B37; Secondary 46B40,  47A10, 47A16, 47A35.}
\keywords{Generalized Cesàro operator, Compactness, Spectra, Power boundedness, Uniform Mean Ergodicity, Sequence space}

\address{ Angela A. Albanese\\
Dipartimento di Matematica ``E.De Giorgi''\\
Universit\`a del Salento- C.P.193\\
I-73100 Lecce, Italy}
\email{angela.albanese@unisalento.it}

\address{Jos\'e Bonet \\
Instituto Universitario de Matem\'{a}tica Pura y Aplicada
IUMPA \\
Universidad Politécnica de Valencia \\
E-46071 Valencia, Spain} \email{jbonet@mat.upv.es}

\address{Werner J.  Ricker \\
Math.-Geogr. Fakultät \\
 Katholische Universität
Eichst\"att-Ingol\-stadt \\
D-85072 Eichst\"att, Germany}
\email{werner.ricker@ku.de}

\begin{abstract} Recent results concerning the linear dynamics and mean ergodicity of compact operators in Banach spaces, together with additional new results, are employed to investigate various spectral properties of generalized Cesàro operators acting in large classes of classical BK-sequence spaces. Of particular interest is to determine the eigenvalues and the corresponding eigenvectors of such operators and to decide whether (or not) the operators are power bounded, mean ergodic and supercyclic.

\end{abstract}
\maketitle


\section{Introduction}

The (discrete) generalized Cesàro operators $C_t$, for $t\in [0,1]$, were first investigated by Rhaly, \cite{R1}. The action of $C_t$  from $\omega:=\C^{\N_0}$ into itself (with $\N_0:=\{0,1,2,\ldots\}$) is given by
\begin{equation}\label{Ces-op}
	C_tx:=\left(\frac{t^nx_0+t^{n-1}x_1+\ldots +x_n}{n+1}\right)_{n\in\N_0},\quad x=(x_n)_{n\in\N_0}\in\omega.
	\end{equation}
For $t=0$ note that $C_0$ is the diagonal operator
\begin{equation}\label{Dia-op}
	D_\varphi x:= \left(\frac{x_n}{n+1}\right)_{n\in\N_0}, \quad x=(x_n)_{n\in\N_0}\in\omega,
	\end{equation}
where $\varphi:=\left(\frac{1}{n+1}\right)_{n\in\N_0}$, and for $t=1$ that $C_1$ is the classical Cesàro averaging operator
\begin{equation}\label{Ces-1}
	C_1x:=\left(\frac{x_0+x_1+\ldots+x_n}{n+1}\right),\quad x=(x_n)_{n\in\N_0}\in\omega.
\end{equation}

The spectra of $C_1$ have been investigated in various Banach sequence spaces. For instance, we mention $\ell^p$ ($1<p<\infty$), \cite{BHS,CR1,G,Le}, $c_0$ \cite{Ak,Le,Re}, $c$ \cite{Le}, $\ell^\infty$ \cite{Le,P,Re}, the Bachelis spaces $N^p$ ($1<p<\infty$) \cite{CR2,CR4}, $bv$ and $bv_0$ \cite{O1,O2}, weighted $\ell^p$ spaces \cite{ABR1,ABR2}, the discrete Cesàro spaces $ces_p$, for $p\in\{0\}\cup (1,\infty)$, \cite{CR3}, and their dual spaces $d_s$ ($1<s<\infty$), \cite{BR1}. For the class of generalized Cesàro operators $C_t$, for $t\in (0,1)$, a study of their spectra and compactness properties (in $\ell^2$) go back to Rhaly, \cite{R1,R2}. A similar investigation occurs for $\ell^p$ ($1<p<\infty$) in \cite{YD} and for $c$ and $c_0$ in \cite{SEl-S,YM}. The paper \cite{SEl-S} also treats $C_t$ when it acts in $bv_0$, $bv$, $c$, $\ell^1$, $\ell^\infty$ and in the Hahn sequence space $h$. In the recent paper \cite{CR4} the setting for considering the compactness and  spectrum of the operators $C_t$ is a large class of Banach lattices contained in $\omega$, which includes all rearrangement invariant sequence spaces (over $\N_0$ for counting measure), and many others.

Our aim is to study the compactness, the spectrum and the linear dynamics of the generalized Cesàro operators $C_t$, for $t\in [0,1]$, when they act in various Banach sequence spaces $X\subseteq \omega$. Such spaces need not be Banach lattices. When $\omega$ is considered to be equipped with its coordinate-wise convergence topology, it is required that the natural inclusion  $X\subseteq \omega$ is continuous (i.e., $X$ is a \textit{BK-space}).
Given any $t\in [0,1]$ the operator $C_t\colon \omega\to\omega$, as given by \eqref{Ces-op}, is continuous, which is indicated by writing $C_t^\omega$. For all the BK-spaces $X$ that we consider, the generalized Cesàro operator $C_t^\omega$ maps $X$ into $X$, necessarily continuously (denoted by $C_t\in\cL(X)$) and it turns out that $C_t$ is a \textit{compact} operator in $X$; see Sections 4-12. For some of the spaces $X$ this is known and for others it needs to be established. Moreover, the point spectrum of $C_t$ in the considered spaces $X$ is always $\{\frac{1}{n+1}\,:\, n\in\N_0\}$, which is also the point spectrum of $C_t^\omega\in\cL(\omega)$, and the corresponding eigenvectors of $C_t^\omega$ in $\omega$ are precisely known; see \eqref{eq.EigenvalueC}. It is often non-trivial to verify that these eigenvectors actually belong to $X$. The compactness of $C_t\in \cL(X)$ is the crucial property. It allows the general results of \cite[Section 6]{ABR-9}, when specialized from locally convex Hausdorff spaces to Banach spaces, together with new results developed here, to be applied to $C_t$ in order to determine whether (or not) it is power bounded, mean ergodic or supercyclic in $X$; see Sections 2 and 3. The choice of the particular BK-spaces $X$ presented in Sections 4-12 in which $C_t$ acts, for $t\in [0,1]$, is extensive and varied and includes $\ell^p$ ($1\leq p\leq\infty$), the space $c$ and its closed subspace $c_0$, as well as the space $cs$ of all convergent series. Also included are the Bachelis spaces $N^p$ ($1<p<\infty$) arising in classical harmonic analysis, the discrete Cesàro spaces $ces_p$, for $p\in\{0\}\cup (1,\infty]$, and their dual spaces $d_s$ ($1\leq s<\infty$). In addition, the spaces of  $p$-bounded variation $bv_p$, for $p\in\{0\}\cup [1,\infty)$, are treated as well as the generalized Hahn spaces $h_d$. As will become evident, each space $X$ which is considered has its own individual features which need to be invoked in order to verify that the general results of Sections 2 and 3 can be applied to $C_t$ acting in that $X$. There is ample scope for future research to adapt the theory developed here in order to apply it to the operators $C_t$ acting in further classes of BK-spaces $X$.

\markboth{A.\,A. Albanese, J. Bonet and W.\,J. Ricker}%
{\MakeUppercase{Mean ergodic and related properties of generalized Ces\`aro operators}}

\section{Preliminaries and notation}

Given locally convex Hausdorff spaces $X, Y$  we denote by $\cL(X,Y)$ the space of all continuous linear operators from $X$ into $Y$. If $X=Y$, then we simply write $\cL(X)$ for $\cL(X,X)$. Equipped with the topology  of pointwise convergence  $\tau_s$ on $X$ (i.e., the strong operator topology) the locally convex Hausdorff space $\cL(X)$ is denoted by $\cL_s(X)$ and for the topology $\tau_b$ of uniform convergence on bounded sets the locally convex Hausdorff space $\cL(X)$ is denoted by $\cL_b(X)$. The closure of a subset $A\subseteq X$ is denoted by $\ov{A}$.
Denote by $\cB(X)$ the collection of all bounded subsets of $X$ and by $\Gamma_X$ a system of continuous seminorms determing the topology of $X$.
 The identity operator on $X$ is denoted by $I$.  The \textit{dual operator} of $T\in \cL(X)$ is denoted by  $T'$; it acts in the topological dual space $X':=\cL(X,\C)$ of $X$. Denote by $X'_\si$ (resp., by $X'_\beta$) the space $X'$ with the weak* topology $\si(X',X)$ (resp., with the strong topology $\beta(X',X)$); see \cite[\S 21.2]{23} for the definition. It is known that $T'\in \cL(X'_\si)$ and $T'\in \cL(X_\be')$,  \cite[p.134]{24}.
For the general theory of functional analysis and operator theory relevant to this paper see, for example, \cite{Dow,23,24,Y}.

Given a locally convex Hausdorff space $X$ and $T\in \cL(X)$, the resolvent set $\rho(T;X)$ of $T$ consists of all $\lambda\in\C$ such that $R(\lambda,T):=(\lambda I-T)^{-1}$ exists in $\cL(X)$. The set $\sigma(T;X):=\C\setminus \rho(T;X)$ is called the \textit{spectrum} of $T$. The \textit{point spectrum}  $\sigma_{pt}(T;X)$ of $T$ consists of all $\lambda\in\C$ (also called an eigenvalue  of $T$) such that $(\lambda I-T)$ is not injective. An eigenvalue $\lambda$ of $T$ is called (geometrically) \textit{simple} if ${\rm dim} \Ker (\lambda I-T)=1$. The range of $\lambda I-T$ is denoted by ${\rm Im}(\lambda I-T):=\{(\lambda I-T)x\,:\, x\in X\}$.

 A linear map $T\colon X\to Y$, with $X,Y$ locally convex Hausdorff spaces, is called \textit{compact} if there exists a neighbourhood $\cU$ of $0$ in $X$ such that $T(\cU)$ is a relatively compact set in $Y$. It is routine to show that necessarily $T\in \cL(X,Y)$. According to \cite[\S 42.1(1)]{24} the compact operators form a 2-sided ideal in $\cL(X,Y)$.

\begin{prop}\label{P.SimpleFact1} Let $X$ be a locally convex Hausdorff space, $Y$ be a  vector subspace of $X$ and $T\in \cL(X)$  satisfy $T(Y)\subseteq  Y$. 
	\begin{itemize}
		\item[\rm (i)] For the relative topology induced in $Y$ by $X$, the restriction operator $T|_Y\in\cL(Y)$. Moreover, a complex number $\lambda$ satisfies 
		\begin{equation}\label{eq.prop}
	\quad	\lambda\in \sigma_{pt}(T|_{Y};Y) \mbox{ if and only if }
		Y\cap \{x\in X\setminus\{0\}\,:\, Tx=\lambda x\}\not=\emptyset.
		\end{equation}
		In particular, $\sigma_{pt}(T|_{Y};Y) \subseteq  \sigma_{pt}(T;X)$.
		\item[\rm (ii)] Let $Y$ have  a locally convex Hausdorff topology such that the natural inclusion $Y\subseteq X$ is continuous and $T|_Y\in\cL(Y)$. Suppose that $\alpha\in \sigma_{pt}(T';X')$ and there exists $u\in X'$ satisfying  $u|_{Y}\not =0$ and $T'u=\alpha u$. Then  $\alpha\in \sigma_{pt}((T|_{Y})';Y')$.
		\end{itemize}
		\end{prop}
		
		\begin{proof}
			(i)  It is routine to verify that $T|_Y\in \cL(Y)$. To establish \eqref{eq.prop}, suppose first  that $\lambda\in \sigma_{pt}(T|_Y;Y)$. Then  there exists $y\in Y\setminus\{0\}$ such that $(T|_Y)y=\lambda y$. Since also $y\in X\setminus \{0\}$ and $Ty=\lambda y$,  it is clear that $y$ belongs to the right-side of \eqref{eq.prop} which is then necessarily non-empty. Conversely, if the right-side of \eqref{eq.prop} is non-empty, then there exists 
			 $z\in Y\setminus\{0\}$ such that $(T|_Y)z=Tz=\lambda z$, that is, $\lambda\in \sigma_{pt}(T|_Y;Y)$.

			(ii) The assumption $T|_Y\in \cL(Y)$ implies that $(T|_Y)'\in \cL(Y'_\sigma)$. Moreover, $v:=u|_Y\in Y'\setminus\{0\}$ because of the continuity of the natural inclusion $Y\subseteq X$. Since $T(Y)\subseteq Y$, it follows for each $y\in Y$ that 
			\[
			\langle y, (T|_{Y})'v\rangle=\langle (T|_{Y})y,v\rangle=\langle Ty,u\rangle=\langle y, T'u\rangle=\langle y, \alpha u\rangle=\langle y, \alpha v\rangle.
			\]
			It follows that $(T|_{Y})'v=\alpha v$ with $v\not=0$ and so $\alpha\in \sigma_{pt}((T|_Y)';Y')$.
		\end{proof} 
	
	\begin{remark}\label{remark2A}\rm  For $Y$ as in Proposition \ref{P.SimpleFact1}(ii) suppose that the closed graph theorem is valid (eg. $Y$ is a Fr\'echet space for its given topology). Then the condition $T|_Y\in \cL(Y)$ follows automatically. Indeed, consider a net $y_\alpha\to 0$ in $Y$ such that $(T|_Y)y_\alpha\to z$ in $Y$, for some $z\in Y$. Then also $y_\alpha\to 0$ in $X$ (as $ Y\subseteq X$ continuously) and  $Ty_\alpha\to z$ in $X$. Since $T\in \cL(X)$, it follows that $z=0$. Hence, $T|_Y\colon Y\to Y$ is a closed operator and so $T|_Y\in \cL(Y)$.
	\end{remark}

An operator $T\in \cL(X)$, with $X$ a locally convex Hausdorff space, is called \textit{power bounded} if $\{T^n:\, n\in\N\}$
is an equicontinuous subset of $\cL(X)$. Here $T^n:= T \circ ...\circ T$ is the composition of T with itself $n$ times. For a Banach space $X$, this means precisely that
$\sup_{n\in\N}\|T^n\|_{X\to X}<\infty$, where $\|\cdot\|_{X\to X}$ denotes the operator norm in $\cL(X)$.  Given $T\in\cL(X)$,  its sequence of averages
\begin{equation}\label{average}
	T_{[n]} := \frac{1}{n}\sum_{m=1}^nT^m,\quad  n\in \N,
\end{equation}
is called the \textit{Ces\`aro means} of $T$. The operator $T$ is said to be \textit{mean ergodic} (resp., \textit{uniformly mean ergodic}) if $(T_{[n]})_{n\in\N}$
is a convergent sequence in $\cL_s(X)$ (resp., in $\cL_b(X)$). It follows from
\eqref{average} that
$\frac{T^n}{n}= T_{[n]}- \frac{n-1}{n}T_{[n-1]}$,
for $n\geq 2$. Hence,
\begin{equation}\label{eq.Tende0}
\frac{T^n}{n}\to 0, \quad n\to\infty,
\end{equation} 
in $\cL_s(X)$ (resp., in $\cL_b(X)$),
whenever $T$ is
mean ergodic (resp., uniformly mean ergodic). A relevant text is \cite{K}; see also \cite{BJS} for locally convex Hausdorff spaces.

Concerning the dynamics of a continuous linear operator $T$ defined on a separable,	locally convex Hausdorff space X, recall that $T$ is said to be \textit{hypercyclic} if there exists $x\in X$ whose orbit $\{T^nx:\, n\in\N_0\}$ is	dense in $X$. If, for some $x\in X$, the projective orbit $\{\lambda T^nx:\, \lambda\in\C,\ n\in\N_0\}$ is dense	in $X$, then $T$ is called \textit{supercyclic}. Clearly, every hypercyclic operator is also supercyclic.  It is routine to verify if $X$ is non-separable, then \textit{no} operator in $\cL(X)$ can be supercyclic. As general	references, we refer to \cite{B-M,G-P}.

We now collect some facts required in the sequel.

\begin{prop}\label{P.SimpleFact2} Let $X$ be a locally convex Hausdorff space, $Y$ be a closed subspace of $X$ and $T\in \cL(X)$ satisfy $T(Y)\subseteq  Y$. The  following properties are satisfied.
	\begin{itemize}
	\item[\rm (i)] If $T\in \cL(X)$ is compact, then $T|_{Y}\in \cL(Y)$ is compact. 	If, in addition,   $Y$ is  infinite dimensional, then $0\in \sigma(T|_{Y};Y)$.
\item[\rm (ii)] If $T\in \cL(X)$ is power bounded, then so is $T|_{Y}\in \cL(Y)$.
If, in addition,  $X$ is a Banach space, then $\|(T|_{Y})^n\|_{Y\to Y}\leq \|
T^n\|_{X\to X}$ for all $n\in\N$.
\item[\rm (iii)] If $T\in \cL(X)$ is (uniformly) mean ergodic, then so is $T|_{Y}\in \cL(Y)$.
\end{itemize}
\end{prop}

\begin{proof} 
(i) According to Proposition \ref{P.SimpleFact1}(i) the operator $T|_Y\in \cL(Y)$. Since $T\in \cL(X)$ is compact, there exists  a neighbourhood $\cU$ of $0$ in $X$ such that $T(\cU)$ is a relatively compact set in $X$. Then $\cU\cap Y$ is a  neighbourhood of $0$ in $Y$ and 
$$T|_{Y}(\cU\cap Y)=T|_{Y}(\cU)\cap  T|_{Y}(Y)\subseteq T|_{Y}(\cU)\cap Y\subseteq T(\cU)\cap Y$$
 is relatively compact in $Y$. This shows that $T|_Y\in \cL(Y)$ is compact.

If, in addition, $Y$ is  infinite dimensional, then necessarily $0\in {\sigma(T|_{Y};Y)}$. Indeed, if $0\not\in {\sigma(T|_{Y};Y)}$, then $0\in\rho(T|_{Y};Y)$  and hence, $ T|_{Y}\colon Y\to Y$ is a surjective topological isomorphism. Since $T|_Y\in\cL(Y)$ is both compact and a topological isomorphism from $Y$ onto itself, there exists  a   neighbourhood $\cU$  of $0$ in $Y$ such that $(T|_Y)(\cU)$ is a relatively compact neighbourhood  of $0$ in $Y$. It follows that $Y$ is necessarily finite dimensional; see \cite[\S 15.7(1)]{24}. This is a contradiction.

(ii) By assumption the set $\{T^n\,:\, n\in\N\} \subseteq \cL(X)$ is equicontinuous. Hence, for every $p\in \Gamma_X$ there exist $q\in \Gamma_X$, $c>0$ such that
\[
p(T^nx)\leq cq(x),\quad x\in X,\ n\in\N,
\]
\cite[Ch. VIII, Section 3]{Y}.
Accordingly, 
\[
p((T|_{Y})^ny)=p(T^ny)\leq cq(y),\quad y\in Y,
\]
as $Y$ is a subspace of $X$ and $T(Y)\subseteq Y$. Since $\{p|_{Y}\, :\, p\in\Gamma_X\}$ generates the locally convex topology of $Y$ induced by $X$, it follows that $\{(T|_Y)^n\,:\ n\in\N\}$ is equicontinuous in $\cL(Y)$, that is, $T|_Y\in \cL(Y)$ is power bounded.

Suppose that $X$ is a Banach space. Then, for  each fixed $n\in\N$, we have that
\[
\|(T|_{Y})^ny\|=\|T^ny\|\leq \|T^n\|_{X\to X}\|y\|,\quad y\in Y,
\]
from which it is  clear that $\|(T|_{Y})^n\|_{Y\to Y}\leq \|
T^n\|_{X\to X}$.

(iii) Suppose that $T\in \cL(X)$ is mean ergodic. Then, for any fixed $y\in Y$, the sequence $((T|_Y)_{[n]}y)_{n\in\N}=(T_{[n]}y)_{n\in\N}$ converges in $X$ and hence, also in $Y$. This shows that $T|_{Y}\in \cL(Y)$ is mean ergodic.

In the event that $T\in \cL(X)$ is uniformly mean ergodic, there exists $P\in \cL(X)$ such that $T_{[n]}\to P$ in $\cL_b(X)$ as $n\to\infty$. Since $T(Y)\subseteq Y$, also $P(Y)\subseteq Y$. Indeed, $T(Y)\subseteq Y$ implies that $T^m(Y)\subseteq Y$ for every  $m\in\N$ and hence, $T_{[n]}(Y)=\frac{1}{n}\sum_{m=1}^nT^m(Y)\subseteq Y$ for every $n\in\N$. Since $Y$ is a closed subspace of $X$, it follows that $\lim_{n\to \infty}T_{[n]}y=Py\in Y$ for every $y\in Y$.

Finally, the fact that $T_{[n]}\to P$ in  $\cL_b(X)$ as $n\to\infty$ implies that $(T|_{Y})_{[n]}\to P{|_Y}$ in $\cL_b(Y)$ as $n\to\infty$, after noting that $\cB(Y)=\{B\cap Y\,:\, B\in \cB(X)\}$.
\end{proof}

The following result is a consequence of Sine's criterion. It can be found in  \cite[Corollary 5.9]{BJS} and in  \cite[\S 21, Theorem 1.4]{K} for Banach spaces.

\begin{prop}\label{P.Non-ME}Let $X$ be a locally convex Hausdorff space and $T\in \cL(X)$ be a power bounded operator such that $1\not\in\sigma_{pt}(T;X)$ but, $1\in \sigma_{pt}(T';X')$. Then $T$ is not  mean ergodic.
	\end{prop}

	We record from \cite[Proposition 5.1 \& Remark 5.3]{A-M}  the following fact, where $B(0,\delta):=\{z\in\C\, :\, |z|<\delta\}$ for every $\delta>0$ and the bar denotes closure in $\C$.
	
	\begin{prop}\label{P.inclusioneD} Let $X$ be a sequentially complete, barrelled locally convex Hausdorff space and $T\in \cL(X)$.
		 If $\frac{T^n}{n}\to 0$ in $\cL_s(X)$ as $n\to\infty$, then $\sigma(T;X)\subseteq \overline{B(0,1)}$. 
			In particular, if $T$ is power bounded, then $\sigma(T;X)\subseteq \overline{B(0,1)}$.   
	\end{prop}
	
	The following result is a special case of Theorem 6.4 in \cite{ABR-9}; see also its proof, where it is shown that $\|T^n-P\|_{X\to X}\to 0$ for $n\to\infty$ and hence, also the sequence of averages $(T_{[n]})_{n\in\N_0}$ of $(T^n)_{n\in\N_0}$ satisfies $\|T_{[n]}-P\|_{X\to X}\to 0$ for $n\to\infty$.
	
	\begin{theorem}\label{T-Fact-1} Let $X$ be a Banach space and $T\in \cL(X)$ satisfy the  conditions that
		\begin{itemize}
			\item[\rm (i)] $T$ is compact,
			\item[\rm (ii)] $1\in \sigma(T;X)$ and $\sigma(T;X)\setminus\{1\}\subseteq\overline{B(0,\delta)}$ for some $\delta\in (0,1)$, and
			\item[\rm (iii)] $\Ker(I-T)\cap {\rm Im}(I-T)=\{0\}$.
		\end{itemize}
		Then $T$ is power bounded, uniformly mean ergodic and  $\|T_{[n]}-P\|_{X\to X}\to 0$ for $n\to\infty$, where $P$ is the projection onto $\Ker(I-T)$ along ${\rm Im}(I-T)$.
		\end{theorem}
	
Condition (ii) in Theorem \ref{T-Fact-1} is equivalent to the requirement that $\sigma(T;X)\subseteq\mathcal{U}:=\{z\in\C:\ |z|\leq 1\}$ and $\sigma(T;X)\cap\partial\mathcal{U}=\{1\}$, where $\partial\mathcal{U}$ denotes the boundary of $\mathcal{U}$. Condition (iii) in Theorem \ref{T-Fact-1} means that the eigenvalue $1$ of $T$ is \textit{semisimple}, that is, 
that $1$ is a first order pole of the resolvent of $T$.

Recall that $\omega=\C^{\N_0}$ is the Fr\'echet space of all complex sequences $x=(x_n)_{n\in\N_0}$ whose locally convex topology is generated by the increasing sequence of seminorms
\begin{equation}\label{eq.sem-omega}
	r_n(x)=\max_{0\leq j\leq n}|x_j|,\quad x=(x_n)_{n\in\N_0}\in\omega,
	\end{equation}	
	for each $n\in\N_0$. Given $x=(x_n)_{n\in\N_0}\in\omega$ we write $x\geq 0$ if $x=|x|$, where $|x|:=(|x_n|)_{n\in\N_0}$. By $x\leq z$ is meant that $(z-x)\geq 0$. Observe that $r_n(x)=r_n(|x|)\leq r_n(|y|)=r_n(y)$ whenever $x,y\in\omega$ satisfy $|x|\leq |y|$. The linear subspace of $\omega$ consisting of all vectors with only finitely many non-zero coordinates is denoted by $c_{00}$; it is dense in $\omega$.
		For each $r\in\N_0$ let $f_r\colon\omega\to\C$ be the linear functional  determined by $(\delta_{k,r})_{k\in\N_0}$, where $\delta_{r,r}=1
	$ and $\delta_{k,r}=0$ if $k\not=r$, that is, $f_r(x):=\langle x, (\delta_{k,r})_{k\in\N_0}\rangle=x_r$, for $x\in\omega$. Of course $f_r\in \omega'_\beta$ for each $r\in\N_0$. In particular, 
	 $\omega'_\beta={\rm span }\{f_r\,:\, r\in\N_0\}$. The set of vectors $\cE:=\{e_n\,:\ n\in\N_0\}$, where $e_n$ has $1$ in coordinate $n$ and $0$ elsewhere, is an (unconditional) Schauder basis for $\omega$. Of course, $c_{00}={\rm span}(\cE)$.
	
	A \textit{sequence space} is a linear subspace $X$ of $\omega$. It is a \textit{Banach sequence space} if it is endowed with a norm $\|\cdot\|_X\colon X\to [0,\infty)$ for which it is complete.
	In the event that the natural inclusion $X\subseteq \omega$ is \textit{continuous}, the Banach sequence space $(X,\|\cdot\|_X)$ is called a \textit{BK-space} (see \S 4.2 in \cite{Wi}, for example). This is equivalent to the continuity of the coordinate projections $P_n\colon X\to \C$, for each $n\in\N_0$, where
	\[
	P_n(x):=x_n,\quad x=(x_n)_{n\in\N_0}\in X,
	\]
	\cite[\S 4.0 (2)]{Wi}.
	Observe, for each $n\in\N_0$,  that the restriction $f_n|_{X}\colon X\to\C$ of $f_n$ from $\omega $ to $X$,  given by 
	\[
	f_n|_X\colon x\to x_n, \quad x\in X,
	\]
	coincides with $P_n$. By a \textit{basis} in a Banach space, we always mean a Schauder basis; see Definition 18.1 in \cite{Dow}, for example.
	
	\begin{lemma}\label{L-C} Let $X\subseteq \omega$ be a BK-space. Then $	f_n|_X\in X'$ for every $n\in\N_0$.
	\end{lemma}
	
	\begin{proof}
	This is a direct consequence of the fact that $X\subseteq \omega$ continuously.
	\end{proof}

Depending on the space $X$, it can happen that $f_n|_X=0$ for various $n\in\N_0$, that is, $\Ker(f_n|_X)=X$.

\begin{example}\label{E-S1}\rm  (i) Let $F$ be any non-empty, finite subset of $\N_0$. Define
	\[
	X:=\{x\in\omega\,:\, x_n=0\ {\rm for\ all \ } n\not\in F\}.
	\]
	For the norm $\|x\|_\infty:=\max_{n\in F}|x_n|$, for $x\in X$, it is clear that $X$ is a finite dimensional Banach sequence space. If $m\in\N_0\setminus F$, then $P_m(x)=0$ for $x\in X$ and so $P_m\colon X\to\C$ is surely continuous. If $m\in F$, then 
	\[
	|P_m(x)|=|x_m|\leq \max_{k\in F}|x_k|=\|x\|_\infty,\quad x\in X,
	\]
	and so again $P_m$ is continuous. Hence, $X$ is a BK-space. Note that $f_n|_X\in X'\setminus\{0\}$ if and only if $n\in F$. 	Since $X={\rm span}\{e_n\}_{n\in F}$, we see that $\cE\not\subseteq X$.

	(ii) Let $X={\rm span}\{\mathbbm{1}\}$ be equipped with the norm
	\[
	\|x\|_X:=\|\alpha \mathbbm{1}\|_X=|\alpha|,\quad x=\alpha \mathbbm{1}\in X\ ({\rm  with\ } \alpha\in\C),
	\]
	 where $\mathbbm{1}=(1,1,\ldots)$. Then $X$ is a BK-space, $e_n\not\in X$ for every $n\in\N_0$, and $f_n|_X\not=0$ for all $n\in\N_0$. Actually, $f_n|_X=\xi$ for every $n\in\N_0$, where
	\[
	\xi(x):=\xi(\alpha \mathbbm{1})=\alpha,\quad x=\alpha \mathbbm{1}\in X\ ({\rm  with\ } \alpha\in\C).
	\]
	(iii) Let $\Gamma:=\{2n+1\,:\, n\in\N_0\}$ and define
	\[
	X:=\{x\in\omega\,:\, x_n=0\ {\rm for\ all\ } n\in\Gamma \ {\rm and } \sup_{n\not\in\Gamma}|x_n|<\infty\}.
	\]
	For the norm
	\[
	\|x\|_\infty:=\sup_{n\in\N_0}|x_n|=\sup_{n\not\in \Gamma}|x_n|,\quad x\in X,
	\]
	it is clear that $X$ is a closed subspace of $\ell^\infty$ and so $X$ is a BK-space. Observe that $\cE\not\subseteq X$ but, $e_n\in X$ for all $n\not\in \Gamma$. Moreover, $f_n|_X=0$ for $n\in\Gamma$ and $f_n|_X=P_n\not=0$ for $n\not\in \Gamma$.
	\end{example}
	
	\section{General results for the generalized Cesàro operators $C_t$}
	
	It was shown in Theorem 6.6 of \cite{ABR-9}, for each $t\in [0,1)$ and for each Banach space $X\in\{d_p,\ell^p\,:\, 1\leq p<\infty\}\cup\{ces_p\,:\, 1<p<\infty\}$, that the generalized Cesàro operator $C_t\in \cL(X)$ is power bounded and uniformly mean ergodic, but not supercyclic. Our aim  is to extend these properties of $C_t$ to a \textit{significantly larger} class of Banach sequence spaces $X$. The following result is Theorem 6.1 in \cite{ABR-9}; see also its proof. We introduce the notation
	\[
	\Lambda:=\left\{\frac{1}{n+1}\, :\ n\in\N_0\right\} \mbox{ and } \Lambda_0:=\Lambda\cup\{0\}.
	\]
	
	\begin{theorem}\label{T-Fact 2} For each $t\in [0,1)$ consider the Cesàro operator $C_t^\omega\in \cL(\omega)$.
		\begin{itemize}
			\item[\rm (i)] $\Ker(I-C_t^\omega)={\rm span}\{x^{[0]}_t\}$ and ${\rm Im}(I-C_t^\omega)=\overline{{\rm span}\{e_r\,:\, r\in\N\}}$, where $x^{[0]}_t:=(t^n)_{n\in\N_0}$.
			\item[\rm (ii)] $\Ker(I-C_t^\omega)\cap {\rm Im}(I-C_t^\omega)=\{0\}$.
			\item[\rm (iii)] $\sigma_{pt}((C_t^\omega)';\omega_\beta')=\sigma((C_t^\omega)';\omega'_\beta)=\Lambda$.
			Moreover, 
			\end{itemize}
			\[
			(C_t^\omega)'z_t^{[n]}=\frac{1}{n+1}z^{[n]}_t,\quad n\in\N_0,
			\]
			where
			\begin{equation}\label{eq.Eigenvalues}
				z_t^{[n]}:=\sum_{i=0}^n(-1)^i\binom{n}{i}t^if_{n-i}\in \omega'_\beta\setminus\{0\}.
			\end{equation}
		
		\end{theorem}
		
		More interesting examples of BK-spaces than those of Example 	\ref{E-S1} and
		which are relevant for the operators $C_t$, $t\in [0,1)$, will come soon.
		
		Let $X\subseteq \omega$ be a Banach sequence space and $t\in [0,1]$. We say that \textit{$C_t$ exists in $X$} if $C_t^\omega(X)\subseteq X$. For each $x\in X$ it is clear from \eqref{Ces-op} that $C_tx\in \omega$. What is required is that actually  $C_tx\in X$.
		
		\begin{remark}\label{R-Uguale} \rm Suppose that $C_t$ exists in $X$. Then  the restriction of $C_t^\omega$ to $X$, denoted by
			\begin{equation}\label{eq.tre}
				C_t^\omega|_X=C_t,
			\end{equation}
			satisfies $C_t^\omega x=C_tx$ for all $x\in X$.
			\end{remark}
		
		The following observation follows from Remark \ref{remark2A}.
		
		\begin{lemma}\label{L-1-Sez2} Let $X\subseteq \omega$ be a BK-space.
			Whenever $C_t$ exists in $X$ for some $t\in [0,1]$, then actually $C_t\in \cL(X)$.
		\end{lemma}

		In view of Lemma \ref{L-1-Sez2} we write $C_t\in \cL(X)$ to mean that $C_t$ exists in $X$.
		
		The norm $\|\cdot\|$ in a Banach sequence space $X\subseteq \omega$ is said to be a \textit{Riesz norm} if $\|x\|\leq \|y\|$ whenever $x,y\in X$ satisfy $|x|\leq |y|$ in $\omega$. If, in addition, $|x|\in X$ whenever $x\in X$, then $\|\cdot\|$ is called an \textit{absolute Riesz norm} in $X$. In this case, $\|x\|=\||x|\|$ for every $x\in X$. Indeed, given $x\in X$, note that $|x|\leq |\, |x| \,|$ and so $\|x\|\leq \||x|\|$. On the other hand, since $ |\, |x| \,|\leq |x|$, also $\||x|\|\leq \|x\|$. Hence, $\|x\|=\||x|\|$. 
		A Banach sequence space $X\subseteq \omega$ is said to be \textit{solid} (or an order ideal) if, whenever $x\in\omega$ and $y\in X$ satisfy $|x|\leq |y|$, then $x\in X$. In this case, $y\in X$ implies that $|y|\in X$. Indeed, $x:=|y|\in \omega$ satisfies $|x|\leq |y|$ and so $x\in X$, that is, $|y|\in X$.
		The norm in the BK-space $cs$ (see Section 6) is \textit{not} an absolute Riesz norm. For instance, $x=\left(\frac{(-1)^n}{(n+1)}\right)_{n\in\N_0}\in cs$, but $|x|\not\in cs$.  Actually, $\|\cdot\|_{cs}$ is not even a Riesz norm. Indeed, for $z=\left(\frac{1}{(n+1)^2}\right)_{n\in\N_0}\in cs$, note that $|z|\leq |x|$ in $\omega$, but $\|z\|_{cs}=\frac{\pi^2}{6}>\log 2=\|x\|_{cs}$.

		The following result implies under certain conditions that $C_t$ exists in $X$ for every $t\in [0,1)$.
		
		\begin{prop}\label{P-Nuova-1} Let $X\subseteq \omega$ be a Banach sequence space which is solid and has an absolute Riesz norm.
		If $C_1\in \cL(X)$, then also $C_t\in \cL(X)$ and $\|C_t\|_{X\to X}\leq \|C_1\|_{X\to X}$ for every $t\in [0,1)$.
			\end{prop}
		
		\begin{proof} Fix $t\in [0,1)$. Then,  for each $x\in X$, we have from \eqref{Ces-op}, \eqref{Ces-1} and the definition of $|\cdot|$ in $\omega$ that
		$|C_tx|\leq C_t|x|\leq C_1|x|$. Since $|x|\in X$, we have that $C_1|x|\in X$.
		By the assumptions on the norm $\|\cdot\|$ and the fact that $X$ is solid,  it follows, for each $x\in X$, that
		\[
		\|C_tx\|=\||C_tx|\|\leq \|C_1|x|\|\leq \|C_1\|_{X\to X}\||x|\|=\|C_1\|_{X\to X}\|x\|.
		\]
		This yields $C_t\in \cL(X)$ with $\|C_t\|_{X\to X}\leq \|C_1\|_{X\to X}$.
		\end{proof}
		
		Let $X\subseteq \omega$ be a Banach sequence space with $c_{00}\subseteq X$ and having an absolute Riesz  norm $\|\cdot\|$. Then $X$ is a BK-space. Indeed, fix $n\in\N_0$. Since $|x_ke_k|\leq |x|$, for each $0\leq k\leq n$ and $x\in X$, it follows that $|x_k|\cdot\|e_k\|=\|x_k e_k\|\leq \|x\|$, that is, $|x_k|\leq (1/\|e_k\|)\|x\|$. Then \eqref{eq.sem-omega} yields that $r_n(x)\leq (\max_{0\leq k\leq n}\frac{1}{\|e_k\|})\|x\|$, for every $x\in X$. Since $n\in\N_0$ is arbitrary, the inclusion $X\subseteq \omega$ is continuous and so $X$ is a BK-space.
		
		The closed unit ball of a Banach space $X$ is denoted by $\cU_X$.
		A partial converse of the previous result is the following one.
		
		\begin{prop}\label{P-Nuova-2} Let $X\subseteq \omega$ be a BK-space such that  $\cU_X$ is compact  (equivalently, closed) in $\omega$. Suppose that $C_t\in \cL(X)$ for all $t\in [0,1)$ and 
			\[
			M:=\sup_{t\in [0,1)}\|C_t\|_{X\to X}<\infty.
			\]
			Then $C_1$ exists in $X$ and $\|C_1\|_{X\to X}\leq M$.
			\end{prop}
		
		\begin{proof} By assumption $\|C_tx\|\leq M\|x\|$ for each $x\in X$ and $t\in [0,1)$. Moreover, for each $n\in\N_0$, it follows from \eqref{Ces-op} and \eqref{Ces-1} that
			\[
			\lim_{t\to 1^-}(C_tx)_n=(C_1x)_n, \ \ \forall x\in X.
			\]
			Clearly, this implies that 
			$
			\lim_{t\to 1^-}C_tx=C_1x$ in $\omega$ for every $x\in X$; see \eqref{eq.sem-omega}.
			
			Fix $x\in \cU_X$, in which case   $C_tx\in M\cU_X$ for each $t\in [0,1)$. Since  $\cU_X$ (hence, also $M\cU_X$) is  compact in $\omega$, it follows that
			\[
			C_1x=\lim_{t\to 1^-}C_tx\in M \overline{\cU_X}^\omega=M\cU_X\subseteq X,
			\]
			where the limit is taken in $\omega$. Therefore, $\|C_1x\|\leq M\|x\|\leq M$. 
			It follows that $C_1\in \cL(X)$ and $\|C_1\|_{X\to X}\leq M$.
			
			Since $\omega$ is a Montel space, \cite[p.155 and \S 27.2]{23}, and $\cU_X$ is a bounded set in $\omega$ (as $X$ is a BK-space), the set $\cU_X$ is compact in $\omega$ if and only if it is closed in $\omega$.
			\end{proof}
		
		The space 
		\[
		d_1:=\left\{x\in\ell^\infty\, :\, \hat{x}:=(\sup_{k\geq n}|x_k|)_{n\in\N_0}\in\ell^1\right\}
		\]
		is a Banach lattice for the norm $\|x\|_{d_1}:=\|\hat{x}\|_1$ and the order induced from $\omega$, \cite{Be}. Since $|x|\leq |\hat{x}|$ for $x\in\ell^\infty$, it is clear that $\|x\|_1\leq \|{x}\|_{d_1}$, that is, $d_1\subseteq \ell^1\subseteq \omega$ with continuous inclusions.
		For each $t\in [0,1)$ define the following sequence of vectors $x_t^{[m]}$, for every   $m\in\N_0$, which are known to  belong to $d_1$, \cite[Lemma 3.6(ii)]{CR4};
		\begin{equation}\label{eq.EigenvalueC}
			\left\{\begin{array}{ll}
				 x_t^{[0]}&=(t^n)_{n\in\N_0}=(1,t,t^2,\ldots)\\
				(x_t^{[m]})_j&=0\ {\rm if\ } 0\leq j<m,\ (x_t^{[m]})_m=1,\ (x_t^{[m]})_{m+n}=\frac{\prod_{i=1}^n(m+i)}{n!}t^n,\ n\geq 1.
				\end{array}\right.
			\end{equation}
Given any $t\in [0,1)$ the operator $C_t^\omega\in \cL(\omega)$ satisfies (cf. \cite[Lemma 3.6(i)]{CR4})
\begin{equation}\label{eq.AutovaloriC}		
	C_t^\omega x_t^{[m]}=\frac{1}{(m+1)}x_t^{[m]},\quad m\in\N_0.
	\end{equation}

	\begin{lemma}\label{L-2Sez2} Let $X\subseteq \omega$ be a BK-space and $t\in [0,1)$ satisfy $C_t\in \cL(X)$.
		\begin{itemize}
			\item[\rm (i)] $\sigma_{pt}(C_t;X)=\{\frac{1}{m+1}\,:\, m\in\N_0  \mbox{ and } x_t^{[m]}\in X\}\subseteq \Lambda$.
			\item[\rm (ii)] Suppose that $C_t\in \cL(X)$ is a compact operator.
			\begin{itemize}
				\item[\rm (a)] If $\sigma_{pt}(C_t;X)$ is a finite set and $0\not\in \sigma(C_t;X)$, then $\sigma(C_t;X)=\sigma_{pt}(C_t;X)$.
				\item[\rm (b)] If $\sigma_{pt}(C_t;X)$ is an infinite set, then $\sigma(C_t;X)=\sigma_{pt}(C_t;X)\cup\{0\}$.
			\end{itemize}
		\end{itemize}
		\end{lemma}
	
		\begin{proof}
			(i)  Proposition \ref{P.SimpleFact1}(i) and Remark \ref{remark2A} imply (for $X:=\omega$ and $Y$ equal to our  given BK-space $X$ and $T:=C_t^\omega\in \cL(\omega)$ there) that
			\[
			\sigma_{pt}(C^\omega_t|_{Y_\omega};Y_\omega)=\left\{\frac{1}{m+1}\,:\, m\in\N_0\mbox{ and } x_t^{[m]}\in Y_\omega\right\},
			\]
			where $Y_\omega$ denotes $Y$ equipped with relative topology from $\omega$. Since $C_t\in \cL(Y)$, by assumption, and $\sigma_{pt}(C^\omega_t|_{Y_\omega};Y_\omega)=\sigma_{pt}(C_t;Y)$, the proof is complete.
			
			(ii)(a) By part (i), $0\not\in \sigma_{pt}(C_t;X)$.
			Since $C_t\in\cL(X)$ is compact,  every non-zero point in the spectrum of $C_t$ is an eigenvalue. By assumption $0\not\in \sigma(C_t;X)$. It follows that $\sigma(C_t;X)=\sigma_{pt}(C_t;X)$.
			
			(b) Part (i) implies that $0$ is a limit point of $\sigma_{pt}(C_t;X)$.   From the theory of compact operators it follows  that $\sigma(C_t;X)=\sigma_{pt}(C_t;X)\cup\{0\}$.
		\end{proof}
		
		\begin{example}\label{Esempio2}\rm (i) Let $m_1,\ldots, m_k$ be distinct numbers in $\N_0$ with $k\geq 1$. Fix any $t\in [0,1)$. Let $X:={\rm span}\{x_t^{[m_1]},\ldots, x_t^{[m_k]}\}\subseteq \ell^1$. Since coordinate $m_j$ of $x_t^{[m_j]}$ is $1$, for each $j\in \{1,\ldots,k\}$, it is routine to check that $X$ is a $k$-dimensional subspace of $\ell^1$. Equipped with the norm $\|\cdot\|_{\ell^1}$ it is clear that $X$ is a closed subspace of $\ell^1$. Since $\ell^1\subseteq \omega$ continuously, also $X\subseteq \omega$ continuously. So, $X$ is a BK-space.
			
			It follows from \eqref{eq.tre} and \eqref{eq.AutovaloriC} that $C_t(X)\subseteq X$, that is, $C_t\in \cL(X)$. Since ${\rm dim}(X)=k$ and $\frac{1}{1+m_j}$, for $1\leq j\leq k$, are distinct eigenvalues of $C_t$ (see Lemma \ref{L-2Sez2}(i)) it follows that 
			\[
			\sigma(C_t;X)=\sigma_{pt}(C_t;X)=\left\{\frac{1}{1+m_j}\,:\, 1\leq j\leq k\right\}.
			\]
			In this case, $0\not\in \sigma(C_t;X)$ and $C_t$ is compact (as ${\rm dim}(X)<\infty$).
			
			(ii) Fix any $r\in\N_0$. Let $X$ denote the \textit{closure}, in $\ell^\infty$ say, of ${\rm span}\{x_t^{[m]}\,:\, m>r\}$, where $t\in [0,1)$ is fixed. Since $\ell^\infty\subseteq \omega$ continuously also $X\subseteq \omega$ continuously. Hence, $X$ is a BK-space. Since $x_t^{[m]}$ has coordinate $m$ equal to $1$ for all $m\geq 0$ and, for $m\geq 1$, all coordinates $j$ for $0\leq j<m$ are $0$ (see \eqref{eq.EigenvalueC}), it follows that $\{x_t^{[m]}\,:\, m>r\}$ is a linearly independent set in $\ell^\infty$ and hence, $X$ is infinite dimensional. Denoting $C_t\in \cL(\ell^\infty)$ by $C_t^\infty$ it is known that  $\|C_t^\infty\|_{\ell^\infty\to \ell^\infty}=1$ and $C_t^\infty\in \cL(\ell^\infty)$ is compact; see Corollary 6.1 and Lemma 6.1 in \cite{SEl-S}. Since $C_t^\infty({\rm span}\{x_t^{[m]}\,:\, m>r\})\subset X$ and the operator $C_t^\infty\in \cL(\ell^\infty)$, it follows that $C_t^\infty(X)\subseteq X$ as ${\rm span}\{x_t^{[m]}\,:\, m>r\}$ is dense in $X$. Hence, the restriction $C_t$ of $C^\infty_t$ to $X$ satisfies $C_t\in \cL(X)$. Moreover, since $C_t^\infty\in \cL(\ell^\infty)$ is compact and $X$ is a $C_t^\infty$-invariant subspace of $C_t^\infty\in \cL(\ell^\infty)$, also $C_t\in \cL(X)$ is compact; see Proposition \ref{P.SimpleFact2}(i). Noting that $X\subseteq \{x\in\ell^\infty\,:\, x_j=0\ \forall\, 0\leq j\leq r\}$ it follows that $x_t^{[m]}\in X$ if and only if $m>r$. Then Lemma \ref{L-2Sez2}(i)-(ii)(b) implies that 
			\[
			\sigma_{pt}(C_t;X)=\left\{\frac{1}{m+1}\,:\, m>r\right\}\ {\rm and}\ \sigma(C_t;X)=\sigma_{pt}(C_t;X)\cup\{0\}.
			\]
			Observe, if $r\geq 1$, then $1\not\in \sigma_{pt}(C_t;X)$ and so $1\not\in \sigma(C_t;X)$. Hence, Theorem \ref{T-Fact-1} does \textit{not} apply to $C_t\in \cL(X)$. However, $\|C_t^\infty\|_{\ell^\infty\to\ell^\infty}=1$ implies that $\|C_t\|_{X\to X}\leq 1$ and so $C_t\in \cL(X)$ is power bounded.
			
			It is shown in Section 6 of \cite{SEl-S} that 
				\[
				\sigma_{pt}(C_t^\infty;\ell^\infty)=\Lambda\ \mbox{ and }\ \sigma(C_t^\infty;\ell^\infty)=\Lambda_0.
				\]
				Hence, $1\in \sigma(C_t^\infty;\ell^\infty)$ and $\sigma(C_t^\infty;\ell^\infty)\setminus\{1\}\subseteq \overline{B(0,1/2)}$. For $x\in \Ker (I-C_t^\infty)\cap {\rm Im}(I-C_t^\infty)$ we see (as $\ell^\infty\subseteq \omega)$) that also 
				$x\in \Ker (I-C_t^\omega)\cap {\rm Im}(I-C_t^\omega)$.
				In the proof of Theorem 6.1 in \cite{ABR-9} it is shown (see also Theorem \ref{T-Fact 2} above) that
				\[
				\Ker (I-C_t^\omega)\cap {\rm Im}(I-C_t^\omega)=\{0\}
				\]
				and so $x=0$. Therefore, 
				\[
				\Ker (I-C_t^\infty)\cap {\rm Im}(I-C_t^\infty)=\{0\}.
				\]
				Applying Theorem \ref{T-Fact-1} to $T:=C_t^\infty$ in $X:=\ell^\infty$ we conclude that $C_t^\infty$ is power bounded (which is clear from $\|C_t^\infty\|_{\ell^\infty\to\ell^\infty}=1$) and uniformly mean ergodic. But, an examination of the proof of Theorem 6.2 in \cite{ABR-9} (see also its statement, since $\sigma(C_t;X)\subseteq B(0,\delta)$ for some $\delta\in (0,1)$) reveals that it is shown there that the Cesàro means
				$(C_t^\infty)_{[n]}\to 0$  in $\ \cL_b(\ell^\infty)$ for $n\to\infty$.
				That is, 
				\[
				\|(C_t^\infty)_{[n]}\|_{\ell^\infty\to\ell^\infty}\to 0\ \mbox{ for } n\to\infty.
				\]
				Since $C_t\in\cL(X)$ is the restriction of $C_t^\infty$ to the closed invariant  subspace $X\subseteq \ell^\infty$, by Proposition \ref{P.SimpleFact2}(iii) we can conclude that $C_t\in \cL(X)$ is uniformly mean ergodic.
		\end{example}
	
	\begin{remark}\label{remark3.}\rm  The operator $C^\infty_t$ has the property that its range $\mbox{Im}(C^\infty_t)\subseteq \cap_{p>1}\ell^p$ for every $t\in [0,1)$. Indeed, if $x\in\ell^\infty$, then \eqref{Ces-op} yields 
		\[
		|(C^\infty_tx)_n|\leq \frac{1}{(n+1)}\sum_{i=0}^nt^{n-i}|x_i|\leq \frac{\|x\|_\infty}{(n+1)}\sum_{i=0}^\infty t^i=\frac{\|x\|_\infty}{(n+1)(1-t)},\quad n\in\N_0.
		\]
		Since $(\frac{1}{(n+1)})_{n\in\N_0}\in \cap_{p>1}\ell^p$, the claim is established.
	\end{remark}
 We now present a rather general theorem.
 
 \begin{theorem}\label{T-general}
 	Let $X\subseteq \omega$ be a BK-space and $t\in [0,1)$. Suppose that
 	\begin{itemize}
 		\item[\rm (i)] $x_t^{[0]}=(t^n)_{n\in\N_0}\in X$ and 
 		\item[\rm (ii)] $C_t$ exists in $X$ and $C_t\in \cL(X)$ is a compact operator.
 	\end{itemize}
 Then $C_t$ is power bounded, uniformly mean ergodic and 
 \[
 \|(C_t)_{[n]}-P\|_{X\to X}\to 0\ \mbox{ for } n\to\infty,
  \]
  where $P$ is the projection onto  ${\rm span}\{x_t^{[0]}\}$ along the subspace $(I-C_t)(X)$.
 \end{theorem}
	
	\begin{proof}	Assumption (i) and Lemma \ref{L-2Sez2}(i) imply that $1\in\sigma(C_t;X)$. Since $C_t$ is assumed to be compact, Lemma \ref{L-2Sez2}(ii) implies, for $\delta=\frac{1}{2}$, that
		\[
		\sigma(C_t;X)\setminus \{1\}\subseteq \overline{B(0,1/2)}.
		\]
		Let $x\in \Ker(I-C_t)\cap {\rm Im}(I-C_t)$. Theorem \ref{T-Fact 2}(ii) gives 
		\[
		\Ker(I-C_t^\omega)\cap {\rm Im}(I-C_t^\omega)=\{0\}.
		\]
	Since $x\in X$, it follows from \eqref{eq.tre} that $x=0$ (in $X$, of course). Hence, 
	\[
	\Ker(I-C_t)\cap {\rm Im}(I-C_t)=\{0\}.
	\]
	Applying Theorem \ref{T-Fact-1} to $T:=C_t$ in $\cL(X)$ we can conclude that $C_t$ is power bounded and uniformly mean ergodic. An examination of the proof of Theorem 6.4 in \cite{ABR-9} reveals that actually $\|P-T^n\|_{X\to X}\to 0$ for $n\to\infty$ and hence, also the sequence of averages of $(T^n)_{n\in\N_0}$ satisfies
	\[
	\|P-T_{[n]}\|_{X\to X}=\|P-(C_t)_{[n]}\|_{X\to X}\to 0\ \mbox{ for } n\to\infty.
	\]	\end{proof}

When applying Theorem \ref{T-general} to particular BK-spaces $X$ the difficulty occurs in verifying the conditions (i), (ii).

Concerning the supercyclicity of $C_t$ we first observe the following general fact.

\begin{lemma}\label{L-nuovo} Let $X$ be a locally convex Hausdorff sequence space containing $c_{00}$ such that $X \subseteq \omega$ with a continuous inclusion.
	Suppose that $C_t\in\cL(X)$ for some $t\in [0,1]$. Then $C_t$ is not supercyclic in $X$.
	\end{lemma}

\begin{proof} Theorem 6.1 in \cite{ABR-9} shows that $C_t^\omega\in\cL(\omega)$ is \textit{not} supercyclic in $\omega$. Moreover, $X$ is dense in $\omega$, as it contains $c_{00}$. Now, if
	$C_t\in\cL(X)$ was supercyclic in $X$, then it would also be supercyclic in $\omega$ by the comparison principle, \cite[\S 1.1.1, p.10]{B-M}.  A contradiction.
	\end{proof}

Now, let $X\subseteq \omega$ be a BK-space and suppose that $t\in [0,1]$ has the property that $C_t$ exists in $X$, that is, $C_t\in \cL(X)$. Then its dual operator $C'_t\colon X'\to X'$ is continuous. Given $x\in X$ and $z\in X'$ we have 
\begin{equation}\label{eq.sei}
\langle C_tx,z\rangle=\langle x,C_t'z\rangle.
\end{equation}
That is, given $z\in X'$ there is a unique $y'\in X'$ (denoted by $C_t'z$) satisfying
\[
\langle C_tx,z\rangle=\langle x,y\rangle,\quad \forall x\in X.
\]
Fix $n\in\N_0$ and consider $z_t^{[n]}\in\omega_\beta'$ as defined by \eqref{eq.Eigenvalues}. Restricting $z_t^{[n]}$ to $X$, that is, each $f_{n-i}$ occurring in \eqref{eq.Eigenvalues} is restricted to $X$, it follows from Lemma \ref{L-C} that $z_t^{[n]}\in X'$.  Moreover,  \eqref{eq.sei} implies that
\begin{equation}\label{eq.sette}
	\langle C_tx,z_t^{[n]}\rangle=\langle x,C_t'z_t^{[n]}\rangle, \quad x\in X.
	\end{equation} 
Since $z_t^{[n]}\in\omega'_\beta$ and $C_tx=C_t^\omega x$ (cf. \eqref{eq.tre}), we have that
\[
\langle C_tx,z_t^{[n]}\rangle=\langle C_t^\omega x,z_t^{[n]}\rangle=\langle x,(C_t^\omega)'z_t^{[n]}\rangle=\langle x,\frac{1}{n+1}z^{[n]}\rangle;
\]
see Theorem \ref{T-Fact 2}(iii). By \eqref{eq.sette} we can conclude that
$
\langle x, C_t'z^{[n]}_t\rangle=\langle x,\frac{1}{n+1}z^{[n]}\rangle$, for all $x\in X$,
with $z^{[n]}_t\in X'$. Accordingly,
\[
C_t'z^{[n]}_t=\frac{1}{n+1}z^{[n]}_t.
\]
Under the \textit{assumption} that $z^{[n]}_t\in X'\setminus\{0\}$ this shows that $\frac{1}{n+1}\in\sigma_{pt}(C'_t;X')$ and that $z^{[n]}_t\not=0$ is a corresponding eigenvector.

\begin{theorem}\label{T-2} Let $X\subseteq \omega$ be a BK-space and $t\in [0,1]$. Suppose that
	\begin{itemize}
		\item[\rm (i)] $C_t$ exists in $X$, that is,  $C_t\in\cL(X)$, and 
		\item[\rm (ii)] the set $\{n\in\N_0\,:\, z_t^{[n]}\in X'\setminus\{0\}\}$ contains at least two elements  or the space $X$ contains $c_{00}$.
	\end{itemize}
Then $C_t$ is not supercyclic.
	\end{theorem}

\begin{proof} If $\{n\in\N_0\,:\, z_t^{[n]}\in X'\setminus\{0\}\}$ contains at least two elements, then  the discussion prior Theorem \ref{T-2} shows that the operator $C'_t\in \cL(X')$ has at least two distinct eigenvalues. It follows from Proposition 1.26 in \cite{B-M} that $C_t$ is not supercyclic. 
	
	In the case that  the BK-space $X$ (hence, $X\subseteq \omega$ with a continuous inclusion) contains $c_{00}$,  Lemma \ref{L-nuovo} yields that $C_t$ is not supercyclic.
\end{proof}

To apply Theorem \ref{T-2}, the condition (ii)
needs to be satisfied. Let us discuss the condition in (ii) which requires that ``the set $\{n\in\N_0\,:\, z_t^{[n]}\in X'\setminus\{0\}\}$ contains at least two elements''.
Suppose that $z_t^{[n]}=0$ as an element of $X'$, for some $n\in\N_0$. Since each functional $f_j$, for $j\in\N_0$, is given by  $f_j(x)=x_j$, for $x\in \omega$ (cf. the discussion after \eqref{eq.sem-omega}), it follows from \eqref{eq.Eigenvalues} that
\begin{equation}\label{eq.Otto}
	\sum_{i=0}^n (-1)^i\binom{n}{i}t^i x_{n-i}=0, \quad \forall x\in X,
	\end{equation}
that is,
 $\langle a,x\rangle=0$, for  $x\in X$,
with $a:=((-1)^{n-i}\binom{n-i}{i}t^{n-i} )_{i=0}^n\in\C^{n+1}$. Since $a\not=0$, there exists $u=(u_0,\ldots, u_n)$ in $\C^{n+1}\setminus\{0\}$ such that $ \langle a,u\rangle\not=0$. So, \eqref{eq.Otto} can be satisfied if and only if 
\[
x_0=x_1=\ldots=x_n=0\ \mbox{ for every } x\in X.
\]
Define $X(n):=\{x\in X\, :\, x_j=0 \ \mbox{ for all }\ 0\leq j\leq n\}$, which is a closed subspace of $X$. So, we have established the following result.

\begin{lemma}\label{L-4} Let $X\subseteq\omega$ be a BK-space. Suppose that $t\in [0,1]$ and $n\in\N_0$ satisfy $z_t^{[n]}\in X'$. Then $z_t^{[n]}\not=0$ if and only if $X\not\subseteq X(n)$.
	\end{lemma}

\begin{remark}\label{R-1}\rm Sufficient conditions for  $X\not\subseteq X(n)$ are as follows.
		
	(i) There exists $0\leq j\leq n$ such that $e_j\in X$. For  most of the classical spaces $X$ that we consider it turns out  that actually $\cE\subseteq X$ and so  $X\not\subseteq X(n)$ for all $n\in\N_0$.
	
	(ii) Fix $t\in [0,1]$. Since the $m$-th coordinate of $x^{[m]}_t$ (cf. \eqref{eq.EigenvalueC}) is 1, for every $m\in\N_0$, it is clear, for any $n\in\N_0$, that  $X\not\subseteq X(n)$ whenever there exists $0\leq m\leq n$ such that $x^{[m]}_t\in X$.
\end{remark}

\section{The spaces $\ell^p$ for $1\leq p\leq \infty$}

In this section we consider the operators $C_t$, for $t\in [0,1]$, when they act in the spaces $\ell^p$, for $1\leq p\leq \infty$. Here, and in the subsequent sections, we treat the case $t=1$ separately from $t\in [0,1)$, as the behaviour of $C_1$ is often quite different to that of $C_t$, for $t\in [0,1)$.

{Consider the case $t=1$.} Since $e_0\in\ell^1$, but $C_1e_0=(\frac{1}{n+1})_{n\in\N_0}\not\in \ell^1$, 
the operator $C_1$ does not exist in $\ell^1$.

\begin{prop}\label{P.4.1} {\rm (i)} let $1<p<\infty$. The operator $C_1\in\cL(\ell^p)$ has norm $\|C_1\|_{\ell^p\to\ell^p}=p'$  (with $\frac{1}{p}+\frac{1}{p'}=1$) and spectrum
	\begin{equation}\label{4.1}
	\sigma(C_1;\ell^p)=\left\{z\in \C\, : \, \left|z-\frac{p'}{2}\right|\leq \frac{p'}{2}\right\}.
	\end{equation}
Moreover,  $C_1\in \cL(\ell^p)$ is not compact and not power bounded and fails to be either mean ergodic or supercyclic.

{\rm (ii)} The operator  $C_1\in \cL(\ell^\infty)$ satisfies $\|C_1\|_{\ell^\infty\to\ell^\infty}=1$ and has spectrum
\begin{equation}\label{4.2}
\sigma(C_1;\ell^\infty)=\left\{z\in\C\, :\, \left|z-\frac{1}{2}\right|\leq \frac{1}{2}\right\}.
\end{equation}
Furthermore,  $C_1\in \cL(\ell^\infty)$ is not compact, not mean ergodic and fails to be supercyclic.
\end{prop}

\begin{proof} (i) It is known that $C_1\in\cL(\ell^p)$ satisfies $\|C_1\|_{\ell^p\to\ell^p}=p'$,  \cite[Theorem 326]{HLP}. According to \cite[Theorem 2]{Le} the spectrum of $C_1$ in $\ell^p$ is given by \eqref{4.1}.
In particular, $C_1\in \cL(\ell^p)$ is not compact. Since $\sigma(C_1;\ell^p)\not\subseteq \overline{B(0,1)}$, it follows from Proposition \ref{P.inclusioneD} (with $X=\ell^p$ and $T=C_1$)  that $C_1$ is not power bounded. Moreover, if $C_1$ was mean ergodic in $\ell^p$, then \eqref{eq.Tende0} would be satisfied from which $\sigma(C_1;\ell^p)\subseteq \overline{B(0,1)}$ would follow (cf. Proposition \ref{P.inclusioneD}). Since this is not the case, we can conclude that $C_1\in \cL(\ell^p)$ is not mean ergodic. Lemma \ref{L-nuovo} implies that $C_1$ is not supercyclic in $\ell^p$.

(ii) According to \cite[Theorem 4]{Le} it is known that $C_1\in \cL(\ell^\infty)$ satisfies $\|C_1\|_{\ell^\infty\to\ell^\infty}=1$ (hence, $C_1$ is power bounded) and $\sigma(C_1;\ell^\infty)$ is given by \eqref{4.2}.
So, $C_1\in\cL(\ell^\infty)$ is not compact. Proposition 4.3 of \cite{ABR00} shows that $C_1$ is not mean ergodic in $\ell^\infty$. Since $\ell^\infty$ is non-separable, the discussion prior to Proposition \ref{P.SimpleFact2} reveals that $C_1$ is not supercyclic; see also Lemma \ref{L-nuovo}.
\end{proof}

\begin{prop}\label{P.4.2} Let $0\leq t< 1$, in which case $C_t\in\cL(\ell^p)$ for all $1\leq p\leq\infty$.
	\begin{itemize}
		\item[\rm (i)] For $p=1$ the operator $C_t\in \cL(\ell^1)$ satisfies $\|C_t\|_{\ell^1\to\ell^1}=t^{-1}\log(1/(1-t))$ whenever $t\in (0,1)$, whereas $\|C_0\|_{\ell^1\to\ell^1}=1$. Moreover, $C_t$ is compact in $\ell^1$ with
		\begin{equation}\label{4.3}
		\sigma_{pt}(C_t;\ell^1)=\Lambda \ \mbox{ and }\ \sigma(C_t;\ell^1)=\Lambda_0,
		\end{equation}
	and  $C_t$ is power bounded and uniformly mean ergodic in $\ell^1$ but, not supercyclic.
	\item[\rm (ii)] For $p=\infty$ the operator  $C_t\in \cL(\ell^\infty)$ satisfies $\|C_t\|_{\ell^\infty\to\ell^\infty}=1$ and is a compact operator with 
	\begin{equation}\label{4.4}
	\sigma_{pt}(C_t;\ell^\infty)=\Lambda \ \mbox{ and }\ \sigma(C_t;\ell^\infty)=\Lambda_0.
	\end{equation}
Moreover, $C_t$ is power bounded and uniformly mean ergodic in $\ell^\infty$ but, not supercyclic.
\item[\rm (iii)] For  $1<p<\infty$ each operator $C_t\in \cL(\ell^p)$ satisfies $\|C_t\|_{\ell^p\to \ell^p}\leq p/(p-1)$  and is  compact  with spectra
\begin{equation}\label{4.5}
\sigma_{pt}(C_t;\ell^p)=\Lambda \ \mbox{ and }\ \sigma(C_t;\ell^p)=\Lambda_0.
\end{equation}
Furthermore, $C_t$ is power bounded and uniformly mean ergodic  but, not supercyclic in $\ell^p$.
	\end{itemize}
\end{prop}

\begin{proof} (i) It is known that $C_t\in \cL(\ell^1)$ satisfies $\|C_t\|_{\ell^1\to\ell^1}=t^{-1}\log(1/(1-t))$, for $t\in (0,1)$, and $\|C_0\|_{\ell^1\to\ell^1}=1$; see \cite[Corollary 8.3]{SEl-S}. Furthermore, $C_t$ is  compact in $\ell^1$ with spectra given by \eqref{4.3},
\cite[Theorem 8.4]{SEl-S}. According to \cite[Theorem 6.6]{ABR-9} the operator $C_t\in \cL(\ell^1)$ is power bounded and uniformly mean ergodic but, not supercyclic.

(ii) Note that $x_t^{[0]}=(t^n)_{n\in\N_0}\in \ell^\infty$. Moreover, $C_t\in \cL(\ell^\infty)$ satisfies $\|C_t\|_{\ell^\infty\to\ell^\infty}=1$, \cite[Corollary 6.1]{SEl-S}, and is a compact operator with spectra given by \eqref{4.4}; see
\cite[Lemma 6.1, Theorems 6.1 \& 6.2]{SEl-S}. Hence, Theorem \ref{T-general} with $X=\ell^\infty$ shows that $C_t\in \cL(\ell^\infty)$ is power bounded (also clear from $\|C_t\|_{\ell^\infty\to\ell^\infty}=1$) and uniformly mean ergodic. Since $\ell^\infty$ is non-separable, $C_t$ is not supercyclic in $\ell^\infty$.

(iii) For each $t\in [0,1)$ the operator $C_t\in \cL(\ell^p)$ satisfies $\|C_t\|_{\ell^p\to \ell^p}\leq p/(p-1)$, \cite[Theorem 2]{YD}, and is  compact, \cite[Theorem 6]{YD},  with spectra given by \eqref{4.5}; see
\cite[Theorems 8 \& 10]{YD}. According to \cite[Theorem 6.6]{ABR-9}, the operator $C_t\in \cL(\ell^p)$ is power bounded, uniformly mean ergodic but, not supercyclic.\end{proof}

\section{The spaces $c$ and $c_0$}

An element $x\in\omega$ belongs to $c$ if and only if $\lim_{n\to\infty}x_n$ exists in $\C$. Then 
\[
\|x\|_c:=\sup_{n\in\N_0}|x_n|,\quad x\in c,
\]
is a norm in $c$. 
Moreover, $c_0:=\{x\in c\, :\, \lim_{n\to\infty}x_n=0\}$ is a proper, closed subspace of $c$. Clearly, $\cE\subseteq c_0\subseteq c$ and both $c_0$ and $c$ are BK-spaces, \cite[p.109]{Wi}.

\begin{prop}\label{P.5.1} For $t=1$ the operator norms of $C_1$ are given by
	\[
	\|C_1\|_{c_0\to c_0}=1=\|C_1\|_{c\to c}.
	\]
	Concerning the spectra we have that 
	\[
	\sigma_{pt}(C_1;c_0)=\emptyset \ \mbox{ and }\ \sigma_{pt}(C_1;c)=\{1\},
		\]
		whereas
	\[
 \sigma(C_1;c_0)=\sigma(C_1,c)=\left\{z\in \C\,:\; \left|z-\frac{1}{2}\right|\leq \frac{1}{2}\right\}.
	\]
	Moreover, $C_1$ fails to be either compact, mean ergodic or supercyclic in both $c_0$ and $c$.
	\end{prop}

\begin{proof} It is routine to verify that $\|C_1\|_{c_0\to c_0}=1=\|C_1\|_{c\to c}$ and hence, both $C_1\in \cL(c_0)$ and $C_1\in \cL(c)$ are power bounded. Concerning the spectra we refer to
 \cite[Theorems 3 \& 5]{Le} and \cite[Theorems 1 \& 3]{Re}. In particular, both $C_1\in \cL(c_0)$ and $C_1\in \cL(c)$ fail to be compact. It is known that both $C_1\in \cL(c_0)$ and $C_1\in \cL(c)$ are not mean ergodic, \cite[Theorem 4.3]{ABR00}. Lemma \ref{L-nuovo} shows that neither $C_1\in \cL(c)$ nor $C_1\in \cL(c_0)$ is supercyclic.\end{proof}

\begin{prop}\label{P.5.2} Let $0\leq t<1$. The operator norms of $C_t$ are given by 
	\[
	\|C_t\|_{c_0\to c_0}=1=\|C_t\|_{c\to c}=1.
	\]
	Moreover, $C_t$  is compact in both $c_0$ and $c$ and has spectra
	\[
	\sigma_{pt}(C_t;c_0)=\sigma_{pt}(C_t;c)=\Lambda \ \mbox{ and }\ \sigma(C_t;c_0)=\sigma(C_t;c)=\Lambda_0.
	\]
	Also, $C_t$ is power bounded and uniformly mean ergodic in both $c_0$ and $c$ but, fails to be supercyclic in both $c_0$ and $c$.
	\end{prop}

\begin{proof} Consider first the BK-space $c_0$. According to \cite[Corollary 8.1 \& Lemma 8.1]{SEl-S} the operator $C_t\in\cL(c_0)$ satisfies $\|C_t\|_{c_0\to c_0}=1$ (hence, $C_t$ is power bounded) and is compact. Moreover, Theorem 8.1(2),(3) of \cite{SEl-S} yields
\[
\sigma_{pt}(C_t;c_0)=\Lambda \ \mbox{ and }\ \sigma(C_t;c_0)=\Lambda_0.
\]
Since $x_t^{[0]}\in\ell^1\subseteq c_0$ and $C_t\in\cL(c_0)$ is compact, Theorem \ref{T-general} shows that $C_t$ is uniformly mean ergodic. That $C_t$ is not supercyclic is clear from Lemma \ref{L-nuovo}.

For the BK-space $c$ it is the case that $\|C_t\|_{c\to c}=1$, \cite[Corollary 8.2]{SEl-S}. Moreover, $C_t\in \cL(c)$ is a compact operator, \cite[Theorem 7]{YM}, with spectra
\[
\sigma_{pt}(C_t;c)=\Lambda \ \mbox{ and }\ \sigma(C_t;c)=\Lambda_0,
\]
\cite[Theorems 8.2 \& 8.3(1)]{SEl-S}. The same argument as above for the space $c_0$ shows that $C_t\in \cL(c)$ is uniformly mean ergodic, but not supercyclic.
\end{proof}

\section{The space $cs$}

An element $x\in\omega$ belongs to the space $cs$ of convergent series if and only if $\sum_{n=0}^\infty x_n$ converges in $\C$, \cite[p.240 \& pp.339-340]{DS}. Equipped with the norm
\[
\|x\|_{cs}:=\sup_{n\in\N_0}\left|\sum_{k=0}^nx_k\right|,\quad x\in cs,
\]
$cs$ is a BK-space, $\cE\subseteq cs$ and is a basis for $cs$, \cite[p.59]{Wi}, and $cs\subseteq c_0$.

Since $e_0\in cs$ and $C_1e_0=\left(\frac{1}{n+1}\right)_{n\in\N_0}\not\in cs$, the operator $C_1$ does not exist in $cs$. 

\begin{prop}\label{P.6.1} Let $0\leq t<1$. Then  $\|C_0\|_{cs\to cs}=1$ and 
	\begin{equation}\label{6.1}
		 \|C_t\|_{cs\to cs}=t^{-1}\log (1/(1-t)),\quad t\in (0,1).
	\end{equation}
The operator $C_t\in \cL(cs)$ is compact and has spectra
\begin{equation}\label{6.2}
\sigma_{pt}(C_t;cs)=\Lambda \ \mbox{ and }\ \sigma(C_t;cs)=\Lambda_0.
\end{equation}
Moreover, $C_t$ is power bounded and uniformly mean ergodic in $cs$ but, not  supercyclic.
\end{prop}

\begin{proof}
According to \cite[Corollary 4.1]{SEl-S} the operator $C_t\in \cL(cs)$
 satisfies \eqref{6.1} with 
  $\|C_0\|_{cs\to cs}=1$. Moreover, $C_t$ is compact, \cite[Lemma 4.2]{SEl-S}, with spectra given by \eqref{6.2},
 \cite[Theorem 4.1(1),(2)]{SEl-S}. Since $x_t^{[0]}\in\ell^1\subseteq cs$ and $C_t$ is compact in $cs$, Theorem \ref{T-general} shows that $C_t$ is power bounded and uniformly mean ergodic. Lemma \ref{L-nuovo} implies that $C_t$ is not supercyclic in $cs$.
 \end{proof}

 \section{The Bachelis spaces $N^p$ for $1<p<\infty$}
 
 These spaces, which arise in classical harmonic analysis, \cite{Bac}, are considered in \cite[Section 2]{CR2} and in \cite[Section 6]{CR4} in relation to the operators $C_t$. The Bachelis spaces $N^p$, for $1<p<\infty$, are reflexive BK-spaces and $\cE\subseteq N^p$ is an unconditional basis, \cite[p.25]{CR4}. Moreover, $\ell^1\subseteq N^p\subseteq c_0$ with continuous inclusions and so $N^p\subseteq \omega$ continuously, \cite[p.25]{CR4}.
 
 \begin{prop}\label{P.7.1} Let $t=1$ and $1<p<\infty$. The operator $C_1\in \cL(N^p)$ and has spectra which satisfy
 	\begin{equation}\label{7.1}
 	\sigma_{pt}(C_1;N^p)=\emptyset \ \mbox{ and }\ \sigma(C_1;N^p)\supseteq \left\{z\in\C\,:\, \left|z-\frac{p}{2}\right|\leq \frac{p}{2}\right\}.
 	\end{equation}
 In particular, $C_1$ is not compact in $N^p$. Moreover, $C_1$ is not power bounded and not mean ergodic in $N^p$ and it fails to be supercyclic.
 	\end{prop}
 
 \begin{proof} It is known that $C_1\in \cL(N^p)$ for all $1<p<\infty$ and that \eqref{7.1} is satisfied,
 \cite[Theorem 2.5]{CR2}. In particular, $C_1$ is not compact in $N^p$. Since $\sigma(C_1;N^p)\not\subseteq \overline{B(0,1)}$, it follows from Proposition \ref{P.inclusioneD} that $C_1$ is not power bounded. If $C_1$ was mean ergodic, then $C_1^n/n\to 0$ in $\cL_s(N^p)$, via the discussion after \eqref{eq.Tende0}, and so $\sigma(C_1;N^p)\subseteq \overline{B(0,1)}$ by Proposition \ref{P.inclusioneD}. Since this is not the case,  $C_1\in\cL(N^p)$ fails to be mean ergodic. Lemma \ref{L-nuovo} implies that $C_1$ is not supercyclic in $N^p$.\end{proof}
 
 The situation for $t\in [0,1)$ is quite different.
 
 \begin{prop}\label{P.7.2} Let $0\leq t<1$ and $1<p<\infty$. The operator norm of $C_t$ satisfies
 	\begin{equation}\label{7.2}
 		\|C_t\|_{N^p\to N^p}\leq 1/(1-t).
 	\end{equation}
 Moreover, $C_t$ is compact in $N^p$ and has spectra
 \begin{equation}\label{7.3}
 \sigma_{pt}(C_t;N^p)=\Lambda \ \mbox{ and }\ \sigma(C_t;N^p)=\Lambda_0.
 \end{equation}
In addition, $C_t$ is power bounded and uniformly mean ergodic in $N^p$ but, it is not supercyclic.
 	\end{prop}
 
 \begin{proof} According to Theorem 6.2 of \cite{CR4}, for every $1<p<\infty$ the operator $C_t\in \cL(N^p)$ satisfies \eqref{7.2}, is compact and has spectra given by \eqref{7.3}.
 Since $\ell^1\subseteq N^p$ and $x_t^{[0]}\in\ell^1$, it is clear that $x_t^{[0]}\in N^p$. Hence, Theorem \ref{T-general} yields that $C_t\in\cL(N^p)$ is power bounded and uniformly mean ergodic. Moreover, Lemma \ref{L-nuovo} implies that $C_t$ is not supercyclic in $N^p$.\end{proof}
 
 \section{The Cesàro spaces $ces_p$ for $p\in\{0\}\cup (1,\infty]$}
 
 For each $1<p<\infty$, the Cesàro space
 \[
 ces_p:=\left\{x\in\omega\, :\, \|x\|_{ces_p}:=\left(\sum_{n=0}^\infty\left(\frac{1}{n+1}\sum_{k=0}^n|x_k|\right)^p\right)^{1/p}<\infty\right\}.
 \]
 Equipped with the norm $\|\cdot\|_{ces_p}$ it is known that $ces_p$ is a reflexive Banach sequence space, \cite{Be}. Lemma 4.3 of \cite{CR4} shows that $ces_p\subseteq \omega$ continuously and so $ces_p$ is a BK-space. According to \cite[Proposition 2.1]{CR3} we have $\cE\subseteq ces_p$ is an unconditional basis and hence, $ces_p$ is separable. For further properties of $ces_p$ see \cite{G-E}.
 
 For $p=\infty$, the space
 \[
 ces_\infty:=\left\{x\in\omega\, :\, \left(\frac{1}{n+1}\sum_{k=0}^n|x_k|\right)_{n\in\N_0}\in\ell^\infty\right\}.
 \] 
 Equipped with the norm $\|x\|_{ces_\infty}:=\sup_{n\in\N_0}\frac{1}{n+1}\sum_{k=0}^n|x_k|$, for $x\in ces_\infty$, the space $ces_\infty$ is a Banach sequence space containing $\cE$. Since $ces_\infty$ is non-separable, \cite[Proposition 4.9(i)]{CR4}, it follows that $\cE$ is not a basis for $ces_\infty$.
 
 For $p=0$, define the sequence space
 \[
 ces_0:=\left\{x\in\omega\, :\, \left(\frac{1}{n+1}\sum_{k=0}^n|x_k|\right)_{n\in\N_0}\in c_0\right\}.
 \] 
 Since $ces_0$ is a proper, closed subspace of $ces_\infty$, \cite[Remark 6.3]{CR3}, it is a Banach sequence space for the norm from $ces_\infty$, which we denote by $\|\cdot\|_{ces_0}$. Lemma 4.3 of \cite{CR4} shows that $ces_0\subseteq ces_\infty\subseteq \omega$ with continuous inclusions and so both $ces_0$, $ces_\infty$ are BK-spaces. Moreover, $\cE\subseteq ces_0$ and $\cE$ is an unconditional basis, \cite[Section 6]{CR3}. Hence, $ces_0$ is separable. The dual space $ces_0'=d_1$ with equal norms, \cite[Lemma 6.2]{CR3}; see the discussion after Proposition \ref{P-Nuova-2} for the definition of $d_1$. Furthermore, with equality of norms, the bidual space $ces_0''=d_1'=ces_\infty$, \cite[Remark 6.3]{CR3}.

 \begin{prop}\label{P.8.1} {\rm (i)} Let  $1<p<\infty$. The operator $C_1\in \cL(ces_p)$ has norm $\|C_1\|_{ces_p\to ces_p}=p':=p/(p-1)$ and  spectrum
 	\begin{equation}\label{8.1}
 	\sigma(C_1;ces_p)=\left\{z\in \C\, : \, \left|z-\frac{p'}{2}\right|\leq \frac{p'}{2}\right\}.
 	\end{equation}
 	Moreover, $C_1$ is not compact,  not power bounded and not mean ergodic in $ces_p$ and fails to be supercyclic.
 	
 	{\rm (ii)} For $p=0$ the operator  $C_1\in \cL(ces_0)$ has norm $\|C_1\|_{ces_0\to ces_0}=1$ and so it is power bounded.   Its spectra are given by 
 	\begin{equation}\label{8.2}
 	\sigma_{pt}(C_1;ces_0)=\emptyset\ \mbox{ and }\ \sigma(C_1;ces_0)=\left\{z\in \C\, : \, \left|z-\frac{1}{2}\right|\leq \frac{1}{2}\right\}.
 	\end{equation}
 	Furthermore, $C_1$ is not compact, not mean ergodic and not supercyclic  in $ces_0$.
 	
 	{\rm (iii)} For $p=\infty$ the operator $C_1\in\cL(ces_\infty)$ satisfies $\|C_1\|_{ces_\infty\to ces_\infty}= 1$  and so it is power bounded. Moreover,  $C_1$ is not compact and not mean ergodic in the non-separable space  $ces_\infty$ and fails to be supercyclic.
 	\end{prop}
 
 \begin{proof} (i) It is known (cf. \cite[p.12]{CR4} and the references given there)
 that $C_1\in \cL(ces_p)$ satisfies $\|C_1\|_{ces_p\to ces_p}=p'$ with spectrum given by \eqref{8.1}.
Hence, $C_1$ is not compact in $ces_p$. Moreover, $C_1\in \cL(ces_p)$ is not mean ergodic, not power bounded and not supercyclic, \cite[Proposition 3.7(ii)]{ABR6}.

(ii) Theorem 6.4  (and its proof) in \cite{CR3} show that $C_1\in \cL(ces_0)$ satisfies $\|C_1\|_{ces_0\to ces_0}=1$ and has spectra as given in \eqref{8.2}.
In particular, $C_1$ is power bounded but, not compact in $ces_0$. It is also shown in the proof of Theorem 6.4 in \cite{CR3} that
\[
\sigma_{pt}(C_1';ces_0')=\{1\}\cup\left\{z\in \C\, : \, \left|z-\frac{1}{2}\right|<\frac{1}{2}\right\}.
\]
In particular, $1\in\sigma_{pt}(C_1';ces_0')$. It follows from Proposition \ref{P.Non-ME} that $C_1$ is not mean ergodic in $ces_0$. It was noted above that $ces_0\subseteq \omega$ continuously and so Lemma \ref{L-nuovo} implies that $C_1\in \cL(ces_0)$ is not supercyclic.

(iii) Lemma 4.1(i) and Lemma 4.8 of \cite{CR4} imply that $\|C_1\|_{ces_\infty\to ces_\infty}\leq 1$. The vector $\mathbbm{1}:=(1^n)_{n\in\N_0}$ belongs to $\ell^\infty\subseteq ces_\infty$, satisfies $\|\mathbbm{1}\|_{ces_\infty}=1$ and $C_1\mathbbm{1}=\mathbbm{1}$, which implies that actually $\|C_1\|_{ces_\infty\to ces_\infty}= 1$ and that $1\in \sigma_{pt}(C_1;ces_\infty)$. Since $C_1\in \cL(ces_0)$ is not compact and $ces_0$ is a closed, $C_1$-invariant subspace of $ces_\infty$, also $C_1\in \cL(ces_\infty)$ fails to be compact; see Proposition \ref{P.SimpleFact2}(i).  Since $C_1\in \cL(ces_0)$ is not mean ergodic, neither is $C_1\in\cL(ces_\infty)$; see Proposition \ref{P.SimpleFact2}(iii). Moreover, $ces_\infty\subseteq \omega$ continuously and so Lemma \ref{L-nuovo} yields that $C_1\in\cL(ces_\infty)$ is not supercyclic.\end{proof}

The situation for $t\in [0,1)$ is rather different than for $t=1$.

\begin{prop}\label{P.8.2} Let $0\leq t<1$.
	\begin{itemize}
		\item[\rm (i)] For each $1<p<\infty$ the operator $C_t\in \cL(ces_p)$  satisfies 
		\begin{equation}\label{8.3}
		\|C_t\|_{ces_p\to ces_p}\leq \min\{(1-t)^{-1}, p/(p-1)\},
		\end{equation}
	is compact and has spectra
		\begin{equation}\label{8.4}
		\sigma_{pt}(C_t;ces_p)=\Lambda\ \mbox{ and }\ \sigma(C_t;ces_p)=\Lambda_0.
		\end{equation}
	Moreover, $C_t$ is  power bounded and  uniformly mean ergodic in $ces_p$ but, fails to be supercyclic.
	\item[\rm (ii)] For $p=0$  the operator $C_t\in \cL(ces_0)$ satisfies 
	\begin{equation}\label{8.5}
	\|C_t\|_{ces_0\to ces_0}=1
	\end{equation}
and so it is power bounded. Moreover, $C_t$  is a compact operator in $ces_0$ with spectra
	\begin{equation}\label{8.6}
	\sigma_{pt}(C_t;ces_0)=\Lambda\ \mbox{ and }\ \sigma(C_t;ces_0)=\Lambda_0. 
	\end{equation}
	In addition, $C_t$ is uniformly mean ergodic  in $ces_0$ but,  not supercyclic.  
	\item[\rm (iii)] For $p=\infty$  the operator $C_t\in \cL(ces_\infty)$ has norm
	\begin{equation}\label{8.7}
	\|C_t\|_{ces_\infty\to ces_\infty}=1
	\end{equation}
	 and so it is power bounded. Furthermore, $C_t$ is a compact operator in $ces_\infty$  with spectra
	\begin{equation}\label{8.8}
	\sigma_{pt}(C_t;ces_\infty)=\Lambda\ \mbox{ and }\ \sigma(C_t;ces_\infty)=\Lambda_0.
	\end{equation}
	In addition, $C_t$ is uniformly mean ergodic in $ces_\infty$ but,  fails to be supercyclic.
	\end{itemize}
	\end{prop}

\begin{proof} (i) For the facts that $C_t\in \cL(ces_p)$ is compact, satisfies \eqref{8.3} and has spectra given by \eqref{8.4} see
\cite[Theorem 4.6 \& Remark 4.7(i)]{CR4}. According to \cite[Theorem 6.6]{ABR-9} the operator $C_t\in \cL(ces_p)$ is power bounded, uniformly mean ergodic, but not supercyclic.

(ii) That the operator $C_t\in \cL(ces_0)$ satisfies \eqref{8.5}, see \cite[Remark 4.7(i)]{CR4}.  According to  \cite[Theorerm 4.6]{CR4}, $C_t$ is a compact operator with spectra equal to \eqref{8.6}.
 Since $c_0\subseteq ces_0$ and $x_t^{[0]}\in\ell^1\subseteq c_0$,  Theorem \ref{T-general} implies that $C_t\in \cL(ces_0)$ is uniformly mean ergodic. Lemma \ref{L-nuovo} shows that $C_t$ is not supercyclic.  
 
(iii)  According to \cite[Theorem 4.10]{CR4} the operator $C_t\in \cL(ces_\infty)$ satisfies \eqref{8.7} and is a compact operator with spectra equal to \eqref{8.8}.
 Since $\ell^1\subseteq \ell^\infty\subseteq ces_\infty$ and $x_t^{[0]}\in\ell^1\subseteq ces_\infty$, it follows from Theorem \ref{T-general} that $C_t\in \cL(ces_\infty)$ is uniformly mean ergodic. The non-separability of $ces_\infty$ yields that $C_t$ is not supercyclic in $ces_\infty$; see also Lemma \ref{L-nuovo}.\end{proof}

\section{The dual Cesàro spaces $d_p$ for $1\leq p<\infty$}

The BK-space $d_1$ has already been  defined and shown to satisfy $d_1\subseteq \ell^1\subseteq \omega$ with continuous inclusions. In the previous section it was noted that $d_1=ces_0'$ with equal norms. For each $p\in (1,\infty)$ define
\[
d_p:=\left\{x\in\ell^\infty\, :\, \hat{x}:=\left(\sup_{k\geq n}|x_k|\right)_{n\in\N_0}\in\ell^p\right\}.
\]
Then $d_p$ is a Banach sequence space for the norm
\[
\|x\|_{d_p}:=\|\hat{x}\|_p,\quad x\in d_p,
\]
which is isomorphic to the dual Banach space $(ces_{p'})'$, \cite[Corollary 12.17]{Be}. In particular, $d_p$ is reflexive for each $p\in (1,\infty)$. Since $|x|\leq |\hat{x}|$, it is clear that $\|x\|_p\leq \|\hat{x}\|_p= \|x\|_{d_p}$, for $x\in d_p$, that is, $d_p\subseteq \ell^p\subseteq \omega$ with continuous inclusions. So, $d_p$ is a BK-space. Since $\ell^1\subseteq \ell^p$ for all $1<p<\infty$, it follows that $d_1\subseteq d_p$ continuously. For each $1\leq p<\infty$ it is known that $\cE\subseteq d_p$ is an unconditional basis for $d_p$; see \cite[Proposition 2.1]{BR1} and \cite[Section 6]{CR3}.

For  $t=1$ the operator $C_1$ does not exist in $d_1$, \cite[p.17]{CR4}.

 So, we only consider $p\in (1,\infty)$. The operator $C_1\in \cL(d_p)$ satisfies $\|C_1\|_{d_p\to d_p}=p'$ with spectra
\[
\sigma_{pt}(C_1; d_p)=\emptyset\ \mbox{ and }\ \sigma(C_1; d_p)=\left\{z\in\C\, :\, \left|z-\frac{p'}{2}\right|\leq \frac{p'}{2} \right\},
\]
\cite[Proposition 3.2 \& Corollary 3.5]{BR1}. In particular, $C_1\in\cL(d_p)$ is not compact. According to \cite[Propositions 3.10 \& 3.11]{BR1} the operator $C_1$ is not mean ergodic and not supercyclic in $d_p$. Since $d_p$ is  reflexive  and power bounded operators in reflexive spaces are necessarily mean ergodic, \cite{Lor}, it follows that $C_1$ is not power bounded in $d_p$.

As is  typically the case, the situation for $t\in [0,1)$ is somewhat  different.

\begin{prop}\label{P.9.1} Let $0\leq t<1$.
	\begin{itemize}
		\item[\rm (i)] Let $p=1$. Then  $C_t\in \cL(d_1)$ and its operator norm satisfies $\|C_0\|_{d_1\to d_1}=1$ and \begin{equation}\label{9.1}
			\|C_t\|_{d_1\to d_1}\leq 1/(1-t)^2,\quad t\in (0,1),
			\end{equation}. 
		\item[\rm (ii)] For $1<p<\infty$ the operator norm of $C_t\in \cL(d_p)$ satisfies  $\|C_0\|_{d_p\to d_p}=1$ and
		\begin{equation}\label{9.2}
		\|C_t\|_{d_p\to d_p}\leq \min\{\|\xi\|_p/(1-t),\ (1-t)^{-1-(1/p)} \},\quad t\in (0,1),
		\end{equation}
		where $\xi:=(1/(n+1))_{n\in\N_0}$.
		\item[\rm (iii)] Let $p\in [1,\infty)$. the operator $C_t$ is compact in $d_p$ with spectra
		\begin{equation}\label{9.3}
		\sigma_{pt}(C_t; d_p)=\Lambda\ \mbox{ and }\ \sigma(C_t;d_p)=\Lambda_0.
		\end{equation}
		Moreover,  $C_t$ is power bounded and uniformly mean ergodic in $d_p$ but, not supercyclic.
	\end{itemize}
	\end{prop}

\begin{proof} (i) That the operator $C_t\in \cL(d_1)$ satisfies \eqref{9.1}, see \cite[Remark 4.18(ii)]{CR4}. Also,  $\|C_0\|_{d_1\to d_1}=1$, \cite[Lemma 3.3]{CR4}. 
	
	(ii) That  the operator $C_t\in \cL(d_p)$ satisfies  $\|C_0\|_{d_p\to d_p}=1$ and \eqref{9.2} see
 \cite[Proposition 4.15 \& Theorem 4.19]{CR4}.

(iii) That the operator $C_t\in \cL(d_p)$ is compact with spectra given by \eqref{9.3} see
\cite[Theorem 4.19]{CR4}. According to \cite[Theorem 6.6]{ABR-9} the operator $C_t\in \cL(d_p)$ is power bounded and uniformly mean ergodic, but not supercyclic.\end{proof}

\section{Solid Banach lattices in $\omega$ }

Let $X\subseteq \omega$ be a Banach sequence space with $c_{00}\subseteq X$ and having an absolute Riesz norm. According to the discussion before Proposition \ref{P-Nuova-2}, $X$ is a BK-space. Throughout this section it is  assumed, for the coordinate-wise order induced from $\omega$, that $(X,\|\cdot\|)$ is a \textit{Banach lattice}; see, for example, \cite{M-N} for the theory of such spaces. If $X\subseteq \omega$ is a Banach lattice which is solid, contains $c_{00}$ and the natural inclusion $X\subseteq \omega$ is continuous, then $X$ is called a \textit{solid Banach lattice}; see \cite[Section 2]{CR4}.

A solid Banach lattice $X\subseteq \omega$ is called \textit{translation invariant} if $S(X)\subseteq X$, where $S\colon \omega\to\omega$ is the right-shift operator given by
\begin{equation}\label{10.1}
	Sx:=(0,x_0, x_1, \ldots),\quad x=(x_n)_{n\in\N_0}\in\omega.
	\end{equation}
The operator $S\in \cL(\omega)$, \cite[p.6]{CR4}. The condition $S(X)\subseteq X$ and a closed graph  argument  imply  that necessarily $S\in \cL(X)$, \cite[p.8]{CR4}. For the following result we refer to \cite[Proposition 3.4 \& Theorem 3.7]{CR4}. First we recall some notation. Let $D_\varphi$ be the diagonal operator given by \eqref{Dia-op}. For  $n\in\N_0$ it is clear from \eqref{eq.sem-omega} that 
\[
r_n(D_\varphi x)\leq \|\varphi\|_\infty r_n(x)=r_n(x),\quad x\in \omega,
\]
and hence, $D_\varphi\in \cL(\omega)$. The restriction of $D_\varphi$ to $X$, again denoted by $D_\varphi$, satisfies $D_\varphi\colon X\to X$ with $\|D_\varphi\|_{X\to X}=1$ and is a compact operator, \cite[Lemma 3.3]{CR4}.

\begin{prop}\label{P-10.1} Let $X\subseteq \omega$ be a translation invariant, solid Banach lattice such that $\cE\subseteq X$ and $\cE$ is  a basis for $X$. Suppose, for some $t\in [0,1)$, that the series
	\begin{equation}\label{eq.10.2}
		R_t:=\sum_{n=0}^\infty t^nS^n
	\end{equation}
is convergent in $\cL_s(X)$, in which case $R_t\in \cL(X)$. Then $C_t=D_\varphi R_t$ and so $C_t\in \cL(X)$ is a compact operator. If, in addition, $d_1\subseteq X$, then
\begin{equation}\label{eq.10-3}
	\sigma_{pt}(C_t; X)=\Lambda\ \mbox{ and }\ \sigma(C_t;X)=\Lambda_0.
	\end{equation}
\end{prop}

An immediate consequence is the following result.

\begin{corollary}\label{C-10-2} Let $X\subseteq \omega$ be a translation invariant, solid Banach lattice satisfying $d_1\subseteq X$ and such that $\cE\subseteq X$ and $\cE$ is a basis for $X$. Suppose, for some $t\in [0,1)$, that  the series in \eqref{eq.10.2} is convergent in $\cL_s(X)$. Then $C_t\in \cL(X)$ is power bounded, uniformly mean ergodic, but not supercyclic.
	\end{corollary}
	
	\begin{proof}
		Since $x_t^{[0]}\in d_1\subseteq X$, Theorem \ref{T-general} implies that $C_t\in \cL(X)$ is power bounded and  uniformly mean ergodic. Moreover, the continuity of the natural inclusion $X\subseteq \omega$ (as $X$ is a BK-space) implies, via Lemma \ref{L-nuovo},  that $C_t$ is not supercyclic.
	\end{proof}

It can be checked that each space $X\subseteq \omega$ in the list
\[
 c_0;\ \ell^p\ \mbox{for } p\in [1,\infty];\ ces_q\ \mbox{for } q\in \{0\}\cup (1,\infty],\ d_p\ \mbox{for } p\in [1,\infty);\ N^p\ \mbox{for } p\in (1,\infty)
 \]
is a solid Banach lattice. The spaces $\ell^\infty$, $ces_\infty$ contain $\cE$ but, being non-separable, $\cE$ cannot be a basis. The Banach lattice $c$ is not solid and $cs$ is not even a Banach lattice.

There is a large class of spaces $X$ to which Corollary \ref{C-10-2} applies. Namely, let $X$ be a \textit{Banach  function space} (with the Fatou property) over the measure space $(\N_0,2^{\N_0},\mu)$, where $\mu\colon 2^{\N_0}\to [0,\infty]$ is counting measure, \cite[p.4]{BS}, \cite[Section 2]{CR4}. Then $c_{00}\subseteq X\subseteq \omega$ with continuous inclusions, \cite[Theorems I.1.4 \& I.1.6]{BS}. Moreover, $X$ has an absolute Riesz norm, is a Banach lattice in $\omega$ and is solid, that is, $X$ is a solid Banach lattice. Assume, in addition, that $X$ is also  \textit{rearrangement invariant}, \cite[Section 2]{CR4}; see \cite[Section II.4]{BS} for the general theory of such spaces. Then 
\begin{equation}\label{eq.10.4}
	\ell^1\subseteq X\subseteq \ell^\infty,
	\end{equation}
with both natural inclusions being continuous and having operator norm $1$, \cite[Corollary II.6.8]{BS}. In particular, $d_1\subseteq X$ as $d_1\subseteq \ell^1$ and hence, $\cE\subseteq X$. Moreover, $\cE$ is a basis for $X$ if and only if $c_{00}$ is dense in $X$ if and only if $X$ is separable, \cite[Lemma 2.1]{CR4}. It is known that $X$ is translation invariant and the right-shift operator $S\in \cL(X)$ is an isometry, \cite[p.10]{CR4}. Hence, for \textit{every} $t\in [0,1)$, the series \eqref{eq.10.2} is absolutely convergent in $\cL(X)$ for the operator norm. The class of rearrangement invariant spaces is rather extensive. It includes, for example, the Lorentz spaces $\ell^{p,r}$ with $p,r\in [1,\infty)$, the Lorentz-Zygmund spaces   $\ell^{p,r}(\log \ell)^\alpha$, various Orlicz-sequence spaces and others; see \cite[Remark 3.10(ii)]{CR4} and the references given there.  Due to the failure of \eqref{eq.10.4} the spaces $ces_p$, for $p\in\{0\}\cup (1,\infty)$, are not rearrangement invariant, \cite[Remark 4.7(ii)]{CR4} as is the case for the spaces $d_p$, for $p\in [1,\infty)$, \cite[Remark 4.20(i)]{CR4}. The above discussion in combination with Proposition \ref{P-10.1} and Corollary \ref{C-10-2} yields the following result.

\begin{prop}\label{P.10.3} Let $X\subseteq \omega$ be a separable,  rearrangement invariant Banach function space over $(\N_0, 2^{\N_0},\mu)$. Then  $d_1\subseteq X$ with $\cE\subseteq X$ and $\cE$ is a basis for $X$. Moreover, for every $t\in [0,1)$, the operator $C_t\in \cL(X)$ is compact with spectra given by \eqref{eq.10-3} and $C_t$ is power bounded, uniformly mean ergodic but, not supercyclic.
	\end{prop}

\section{The spaces $bv_0$, $bv$ and $bv_p$ for $ 1<p<\infty$}

We begin by considering the spaces $bv$ and $bv_0$. An element $x\in\omega$ belongs to $bv$ if and only if 
\[
\|x\|_{bv}:=|x_0|+\sum_{k=1}^\infty|x_k-x_{k+1}|<\infty,
\]
equivalently, if and only if $|\lim_{k\to\infty}x_k|+\sum_{k=1}^\infty|x_k-x_{k+1}|<\infty$. Then $bv$ is a Banach sequence space for the norm $\|\cdot\|_{bv}$ and $bv_0=bv\cap c_0$ is a closed, proper subspace of $bv$. Both $bv_0$ and $bv$ are BK-spaces, \cite[pp.109-110]{Wi}. Clearly $\cE\subseteq bv_0\subseteq bv$. Note that $\|\cdot\|_{bv}$  is not  an absolute Riesz norm. Indeed, $x=((-1)^n/(n+1)^2)_{n\in\N_0}$ and $|x|=(1/(n+1)^2)_{n\in\N_0}$ both belong to $bv_0$ but, $\|x\|_{bv}=\pi^2/3>2=\|\,|x|\,\|_{bv}$.

\begin{prop}\label{P.11.1} {\rm (i)} For $t=1$  the operator $C_1\in \cL(bv)$ satisfies $\|C_1\|_{bv\to bv}=1$ and has spectra
	\begin{equation}\label{A}
	\sigma_{pt}(C_1;bv)=\{1\}\ \mbox{ and }\ \sigma(C_1;bv)=\left\{z\in\C\, :\, \left|z-\frac{1}{2}\right|\leq \frac{1}{2}\right\}
	\end{equation}
	The operator  $C_1$  is power bounded in $bv$ but, it is not compact,  not mean ergodic and fails to be  supercyclic.
	
	{\rm (ii)} For $t=1$ the operator $C_1\in \cL(bv_0)$ satisfies $\|C_1\|_{bv_0\to bv_0}=1$ and has spectra
	\begin{equation}\label{B}
	\sigma_{pt}(C_1;bv_0)=\emptyset\ \mbox{ and }\ \sigma(C_1;bv_0)=\left\{z\in\C\, :\, \left|z-\frac{1}{2}\right|\leq \frac{1}{2}\right\}.
	\end{equation}
	 The operator $C_1$ is power bounded in $bv_0$ but, it is not compact, not mean ergodic and fails to supercyclic.
	\end{prop}

\begin{proof} (i) It is known that $C_1\in \cL(bv)$ satisfies $\|C_1\|_{bv\to bv}=1$, \cite[Lemma 2.1]{O2} and has spectra given by \eqref{A},
\cite[Corollaries 2.4 \& 2.6]{O2}. In particular, $C_1\in\cL(bv)$ is not compact, but it is power bounded and $\mathbbm{1}:=(1^n)_{n\in\N_0}\in bv$ is an eigenvector corresponding to the eigenvalue $1$. According to \cite[Proposition 4.7]{ABR00}, $C_1\in \cL(bv)$ is not mean ergodic. Lemma \ref{L-nuovo} implies that $C_1$ is not supercyclic in $bv$.

 (ii) The operator $C_1\in \cL(bv_0)$ satisfies $\|C_1\|_{bv_0\to bv_0}=1$ and has spectra given by \eqref{B}; see \cite[Lemma 1.2]{O1}, \cite[Remark 3.1]{SEl-S}.
So, $C_1\in \cL(bv_0)$ is power bounded, but not compact. According to \cite[Proposition 4.4]{ABR00} the operator $C_1$ is not mean ergodic in $bv_0$. That $C_1$ is not supercyclic in $bv_0$ follows from Lemma \ref{L-nuovo}.\end{proof}

\begin{prop}\label{P.11.2} Let $0\leq t<1$.
	\begin{itemize}
		\item[\rm (i)] The operator norm of $C_t\in \cL(bv)$ satisfies $\|C_t\|_{bv\to bv}=1$ and $C_t$ is compact in $bv$ with  spectra equal to 
		\begin{equation}\label{C}
		\sigma_{pt}(C_t;bv)=\Lambda\ \mbox{ and }\ \sigma(C_t;bv)=\Lambda_0.
		\end{equation}
	Moreover, $C_t$ is power bounded and uniformly mean ergodic in $bv$ but, fails to be supercyclic.
	\item[\rm (ii)] The operator norm of $C_t\in \cL(bv_0)$ satisfies $\|C_t\|_{bv_0\to bv_0}=1$ and  $C_t$ is  compact in $bv_0$  with spectra equal to 
	\begin{equation}\label{D}
	\sigma_{pt}(C_t;bv_0)=\Lambda\ \mbox{ and }\ \sigma(C_t;bv_0)=\Lambda_0.
	\end{equation}
	Furthermore,  $C_t$ is power bounded and uniformly mean ergodic in $bv_0$ but,  it is not supercyclic.
	\end{itemize}
	\end{prop}

\begin{proof} (i)   It is known that $C_t\in \cL(bv)$ satisfies $\|C_t\|_{bv\to bv}=1$, \cite[Corollary 7.1]{SEl-S}, and so $C_t$ is power bounded. The operator $C_t\in \cL(bv)$ is compact, \cite[Lemma 7.2]{SEl-S}, and has spectra given by \eqref{C},
\cite[Theorems 7.1 \& 7.2]{SEl-S}. Moreover, $x_t^{[0]}\in\ell^1\subseteq bv$ and so Theorem  \ref{T-general} yields that $C_t\in \cL(bv)$ is uniformly mean ergodic, whereas Lemma \ref{L-nuovo} shows that $C_t$ is not supercyclic in $bv$.

(ii) Theorem 1.1 of \cite{SEl-S} shows that $C_t\in \cL(bv_0)$ satisfies $\|C_t\|_{bv_0\to bv_0}=1$ and that $C_t$ is a compact operator with spectra given by \eqref{D}.
Proposition \ref{P.SimpleFact2}(iii) implies that $C_t\in \cL(bv_0)$ is uniformly mean ergodic (as does Theorem \ref{T-general} because $x_t^{[0]}\in\ell^1\subseteq bv_0$) and Lemma \ref{L-nuovo} implies that $C_t$ is not supercyclic in $bv_0$.\end{proof}

We now turn to the operators $C_t$ acting in $bv_p$ for $p>1$. So, fix $p\in (1,\infty)$. The linear space
\[
bv_p:=\left\{x\in\omega\, :\, \|x\|_{bv_p}:=\left(|x_0|^p+\sum_{k=1}^\infty |x_{k+1}-x_k|^p\right)^{1/p}<\infty\right\}
\]
equipped with the norm $\|\cdot\|_{bv_p}$ is a BK-space, \cite[Theorem 2.1]{BA}, and satisfies $\cE\subseteq bv_p$. Actually $\ell^p\subseteq bv_p$ with a continuous (proper) inclusion, \cite[Theorem 2.4]{BA}. Since $bv_p$ has a basis, \cite[Theorem 3.1]{BA}, it is necessarily separable. The map $T\colon bv_p\to \ell^p$ defined by 
\begin{equation}\label{eq.11.1}
	T_px:=(x_0,\ x_1-x_0,\ x_2-x_1,\ x_3-x_2,\ldots),\quad x=(x_n)_{n\in\N_0}\in bv_p,
	\end{equation}
is a linear isomorphism from $bv_p$ onto $\ell^p$ with inverse map given by 
\begin{equation}\label{eq.11.2}
	T_p^{-1}y:=(y_0,\ y_0+y_1,\ y_0+y_1+y_2, \ldots),\quad y=(y_n)_{n\in\N_0}\in \ell^p;
\end{equation}
see Theorem 2.2 and its proof in \cite{BA}. In particular, $bv_p$ is reflexive.

For $t=1$ the operator $C_1\in \cL(bv_p)$ satisfies $\|C_1\|_{bv_p\to bv_p}=1$, \cite[Theorem 3.1]{AB}, and has spectra
\[
\sigma_{pt}(C_1;bv_p)=\{1\}\ \mbox{ and }\ \sigma(C_1;bv_p)=\left\{z\in\C\, :\, \left|z-\frac{1}{2}\right|\leq \frac{1}{2}\right\},
\]
\cite[Theorems 3.3 \& 3.4]{AB}. Hence, $C_1$ is not compact in $bv_p$, but it is power bounded. Moreover,  $C_1\in \cL(bv_p)$ is not mean ergodic, \cite[Proposition 4.7]{ABR00}, and (via Lemma \ref{L-nuovo}) not supercyclic.

Concerning the case when  $0\leq t<1$ requires some preparation.  It was shown in Section 10 that the diagonal operator $D_\varphi\in \cL(\omega)$, where $D_\varphi$ is given by \eqref{Dia-op}.

\begin{lemma}\label{L.11.1} For each $p\in (1,\infty)$ it is the case that $D_\varphi(bv_p)\subseteq bv_p$. Hence, the restriction of $D_\varphi$ to $bv_p$, again denoted by $D_\varphi$, exists and belongs to $\cL(bv_p)$. Moreover, $D_\varphi\colon bv_p\to bv_p$ is a compact operator.
\end{lemma} 
	\begin{proof} Fix  $p\in (1,\infty)$. Define a linear map $A$ in $\ell^p$ by
		\[
		Ay:=(\varphi_ny_n)_{n\in\N_0},\quad y=(y_n)_{n\in\N_0}\in\ell^p.
		\]
		Since $X:=\ell^p$ is a solid Banach lattice, it follows from \cite[Lemma 3.3]{CR4} that $\|A\|_{\ell^p\to\ell^p}=1$ and $A\in\cL(\ell^p)$ is a compact operator.
		
		For each $y\in \ell^p$ define $By\in\omega$ by
		\[
		(By)_0:=0 \ \mbox{ and } \ (By)_n:=\frac{1}{n(n+1)}\sum_{j=0}^{n-1}y_j,\ \mbox{ for }\ n\geq 1.
		\]
		The claim is that $B\in \cL(\ell^p)$. To see this, fix $y\in\ell^p$. Since $|y_n|\leq \|y\|_p$ for each $n\in\N_0$, it follows that
		\begin{align*}
			\|By\|_p^p &\leq \frac{1}{(1\cdot 2)^p}|y_0|^p+\frac{1}{(2\cdot 3)^p}(|y_0|+|y_1|)^p+\frac{1}{(3\cdot 4)^p}(|y_0|+|y_1|+|y_2|)^p+\ldots\\
			& \leq \frac{1}{2^p} \|y\|_p^p+\frac{1}{3^p} \|y\|_p^p+\frac{1}{4^p} \|y\|_p^p+\ldots\\
			&=\left(\sum_{k=2}^\infty\frac{1}{k^p}\right) \|y\|_p^p.
		\end{align*}
	This inequality implies that $\|B\|_{\ell^p\to\ell^p}\leq \left(\sum_{k=2}^\infty\frac{1}{k^p}\right)^{1/p}<\infty$.
	
	The next claim is that $B\in \cL(\ell^p)$ is compact. To see this define a finite rank operator $B^{(N)}\in\cL(\ell^p)$, for each $N\geq 3$, by 
		\[
		(B^{(N)}y)_0:=0, \  (B^{(N)}y)_n:=(By)_n\  \mbox{for}\  1\leq n< N \ \mbox{and}\  (B^{(N)}y)_n:=0 \ \mbox{for}\  n\geq N,
		\]
		and every $y\in\ell^p$. Observe that 
		\begin{align*}
			\|(B-B^{(N)})y\|_p^p&=\sum_{n=N}^\infty\left|\frac{1}{n(n+1)}\sum_{j=0}^{n-1}y_j\right|^p\\
			&\leq \sum_{n=N}^\infty\frac{1}{n^p(n+1)^p}\left(\sum_{j=0}^{n-1}|y_j|\right)^p\\
			&\leq \sum_{n=N}^\infty\frac{1}{n^p(n+1)^p} (n\|y\|_p)^p=\left(\sum_{n=N}^\infty\frac{1}{(n+1)^p}\right)\|y\|_p^p.
		\end{align*}
	Accordingly, $\|B-B^{(N)}\|_{\ell^p\to\ell^p}\leq \left(\sum_{n=N}^\infty\frac{1}{(n+1)^p}\right)^{1/p}$ for every $N\geq 3$, which implies that $\|B-B^{(N)}\|_{\ell^p\to\ell^p}\to 0$ for $N\to\infty$ and hence, that $B\in \cL(\ell^p)$ is compact. So, the operator $F:=(A-B)$ is also compact in $\ell^p$. Direct calculation via \eqref{eq.11.1} and \eqref{eq.11.2} shows that $D_\varphi y=(T^{-1}_pFT_p)y$ for every $y\in bv_p\subseteq \omega$. By the ideal property for compact operators it follows that $D_\varphi\in \cL(bv_p)$ is compact.
		\end{proof}
	
	The right-shift operator $S\in \cL(\omega)$ is given by \eqref{10.1}. For $x\in bv_p$ note that
	\[
	\|Sx\|^p_{bv_p}=\|(0, x_0, x_1,x_2,\ldots)\|^p_{bv_p}=0^p+|x_0-0|^p+|x_1-x_0|^p+|x_2-x_1|^p+\ldots=\|x\|_{bv_p}^p.
	\]
	Hence, $S\in \cL(bv_p)$ is an \textit{isometry}. Accordingly, the formal series \eqref{eq.10.2}, when considered in $\cL(bv_p)$, satisfies $\sum_{n=0}^\infty t^n\|S^n\|_{bv_p\to bv_p}=1/(1-t)<\infty$ for each $t\in [0,1)$, that is, the series $R_t=\sum_{n=0}^\infty t^nS^n$ is absolutely convergent in  $\cL_b(bv_p)$.
	
	\begin{prop}\label{P.11.2} For each $p\in (1,\infty)$ and  $t\in [0,1)$ the operator $C_t\in \cL(bv_p)$ is compact  with spectra
		\begin{equation}\label{eq.11.3}
			\sigma_{pt}(C_t;bv_p)=\Lambda\ \mbox{ and }\ \sigma(C_t;bv_p)=\Lambda_0.
			\end{equation}
		Moreover, $C_t\in \cL(bv_p)$ is power bounded, uniformly mean ergodic, but not supercyclic.
		\end{prop}
\begin{proof}
	Fix $p\in (1,\infty)$ and $t\in [0,1)$. By the discussion prior to the proposition, $R_t\in \cL(bv_p)$ and so  Lemma \ref{L.11.1} implies that 
	\begin{equation}\label{eq.11.4}
		D_\varphi R_t=\sum_{n=0}^\infty t^nD_\varphi S^n\in \cL(bv_p)
		\end{equation}
	is a compact operator. According to \cite[Lemma 3.2]{CR4} the identity $D_\varphi R_t=C_t$ holds  in $\cL(\omega)$ and so $C_t^\omega(bv_p)\subseteq bv_p$; see \eqref{eq.11.4}. Lemma \ref{L-1-Sez2} yields that $C_t\in \cL(bv_p)$ with $C_t$ compact, again via \eqref{eq.11.4}.
	
	For each $t\in [0,1)$ the set $\{x_t^{[m]}\, :\, m\in\N_0\}\subseteq d_1\subseteq \ell^p\subseteq bv_p$ and hence, Lemma \ref{L-2Sez2} (with $X=bv_p$) implies that \eqref{eq.11.3} is valid. Furthemore, Theorem \ref{T-general} shows that $C_t\in \cL(bv_p)$ is power bounded and uniformly mean ergodic. However, $C_t$ is not supercyclic; see Lemma \ref{L-nuovo}.\end{proof}
	
	\section{The generalized Hahn spaces $h_d$}
	
	A vector $x\in\omega$ belongs to $h$ if and only if $x\in c_0$ and
	\[
\|x\|_h:=	\sum_{k=0}^\infty (k+1)|x_{k+1}-x_k|<\infty.
	\]
	Let $|||x|||:=\|x\|_h+\|x\|_\infty$ for $x\in h$. H. Hahn, \cite{Ha}, showed that $(h, |||\cdot|||)$ is a BK-space and K.C. Rao, \cite{Ra}, showed that $(h, \|\cdot\|_h)$ is a BK-space with $\cE\subseteq h$ and $\cE$  is a basis for $h$. Since $\|x\|_h\leq |||x|||$ for every $x\in h$, the norms $\|\cdot\|_h$ and $|||\cdot|||$ are equivalent on $h$. The BK-space $(h, \|\cdot\|_h)$ is called the \textit{Hahn space}. This space is well studied; see \cite{MRT} and the references therein.
	
	Since $C_1e_0=(1/(n+1))_{n\in\N_0}\not\in h$, the operator $C_1$ does not exist in $h$. 
	So, we only consider $C_t$ for $t\in [0,1)$. 
	
	\begin{prop}\label{PN.12.1} Let $0\leq t<1$. The operator $C_t\in \cL(h)$ is compact and has spectra
		\begin{equation}\label{E}
		\sigma_{pt}(C_t;h)=\Lambda\ \mbox{ and }\ \sigma(C_t;h)=\Lambda_0.
		\end{equation}
		Moreover,  $C_t$ is power bounded and uniformly mean ergodic in $h$ but, it is not supercyclic.
	\end{prop}
	
\begin{proof}	According to \cite[Corollary 5.1 \& Lemma 5.2]{SEl-S} the operator $C_t\in \cL(h)$ is compact and, by \cite[Theorem 5.1 (1)(2)]{SEl-S}, it has spectra given by \eqref{E}.

	Since $x_t^{[0]}=(t^n)_{n\in\N_0}$ satisfies
	\[
	\|x_t^{[0]}\|_h=\sum_{k=0}^\infty (k+1)|t^{k+1}-t^k|=(1-t)\sum_{k=0}^\infty (k+1)t^k<\infty,
	\]
	it follows from Theorem \ref{T-general} that $C_t\in \cL(h)$ is power bounded and uniformly mean ergodic. Lemma \ref{L-nuovo} yields that $C_t$ is not supercyclic in $h$.\end{proof}
	
	Let $d=(d_n)_{n\in\N_0}$ be a monotone increasing sequence of numbers in $(0,\infty)$ with $d_n\uparrow \infty$. G. Goes, \cite{Go}, introduced the \textit{generalized Hahn space}
	\[
	h_d:=c_0\cap\left\{x\in\omega\, :\, \|x\|_{h_d}:=\sum_{k=0}^\infty d_k|x_{k+1}-x_k|<\infty\right\}.
	\]
	For the norm $\|\cdot\|_{h_d}$ it is known that $h_d\subseteq c_0\subseteq \omega$ is a BK-space with $\cE\subseteq h_d$ and $\cE$ is a basis for $h_d$, \cite[Proposition 2.1]{MRT}. Of course, $h=h_d$ for the particular choice $d=(n+1)_{n\in\N_0}$. It is always assumed that $d$ satisfies $d_n\geq 1$ for all $n\in\N_0$, in which case we write $1\leq d_n\uparrow \infty$. Otherwise set $d_n'={d_n}/{d_0}$ and observe that the identity
	\[
	\sum_{k=0}^\infty d_k|x_{k+1}-x_k|=d_0\sum_{k=0}^\infty d_k'|x_{k+1}-x_k|,\quad x\in\omega,
	\]
	implies $h_d=h_{d'}$ with equivalent norms, where $d'=(d'_n)_{n\in\N_0}$.
	
	Given $1\leq d_n\uparrow \infty$ let
	\[
	bs_d:=\left\{x\in\omega\,:\, \sup_{n\in\N_0}\frac{1}{d_n}\left|\sum_{k=0}^nx_k\right|<\infty\right\}.
	\]
	Equipped with the norm
	\[
	\|x\|_{bs_d}:=\left|\left|\left(\frac{1}{d_n}\left|\sum_{k=0}^nx_k\right|\right)_{n\in\N_0}\right|\right|_\infty,\quad x\in bs_d,
	\]
	it is known that $bs_d$ is a BK-space, \cite[Remark 2.2]{MRT}, which is isometrically isomorphic to the dual Banach space $h_d'$, \cite[Proposition 2.3]{MRT}. The duality between $h_d$ and $bs_d$ is the natural one. Whenever $x\in\ell^1$ we have (as $1\leq d_n$ for all $n\in\N_0$) that
	\[
	\frac{1}{d_n}\left|\sum_{k=0}^n x_k\right|\leq \sum_{k=0}^n|x_k|\leq \|x\|_1,\quad n\in\N_0.
	\]
	Hence, $\ell^1\subseteq h_d'$ continuously.
	
	Observe that $e_0\in h_{d}$ and $x:=C_1^\omega e_0=(1/(n+1))_{n\in\N_0}$ satisfies
	\begin{equation}\label{eq.Nuova}
	\sum_{k=0}^\infty d_k\left|x_{k+1}-x_k\right|=\sum_{k=0}^\infty \frac{d_k}{(k+2)(k+1)}.
	\end{equation}
	So, if $1\leq d_n\uparrow \infty$ has the property that $(\frac{d_n}{(n+2)(n+1)})_{n\in\N_0}\not\in \ell^1$, then \eqref{eq.Nuova} implies that
	 $C_1^\omega e_0\not\in h_d$. In this case $C_1$ does not exist in $h_d$. So, we will concentrate mainly on $C_t$ for $t\in [0,1)$.
	 
	 Let $A=(a_{n,k})_{n,k\in\N_0}$ be an infinite matrix with entries from $\C$ and $X$, $Y$ be Banach sequence spaces in $\omega$. For $x=(x_n)_{n\in\N_0}\in X$ write $A_nx:=\sum_{k=0}^\infty a_{n,k}x_k$, for $n\in\N_0$, and $Ax:=(A_nx)_{n\in\N_0}$, provided all series converge. One seeks conditions on $A$ which ensure that $Ax\in Y$, for every $x\in X$, and the linear map $x\mapsto Ax$, for $x\in X$, belongs to $\cL(X,Y)$, \cite[Sections 1 \& 3]{MRT}. For the choice $X=Y=h_d$ with $1\leq d_n\uparrow \infty$ it is known, \cite[Theorem 3.9]{MRT}, that $A$ induces an operator in $\cL(h_d)$ if and only if  column $k$ of $A$ belongs to $c_0$, for every $k\in\N_0$, and 
	 \[
	 \sup_{m\in\N_0}\left(\frac{1}{d_m}\sum_{n=0}^\infty d_n\left|\sum_{k=0}^m(a_{n,k}-a_{n+1,k})\right|\right)<\infty.
	 \]
	
	Given $t\in [0,1]$ let the matrix $A(t)=(a_{n,k}(t))_{n,k\in\N_0}$ be defined as follows. Namely, for each $n\in\N_0$, set
	\begin{equation}\label{eq.12.1}
		a_{n,k}(t):=\left\{\begin{array}{ll}
			\frac{t^{n-k}}{n+1}, & 0\leq k\leq n,\\
			0, & k>n.
		\end{array}\right.
		\end{equation}
	According to \eqref{Ces-op}, the matrix $A(t)$ represents the generalized Cesàro operator $C_t^\omega\in\cL(\omega)$ relative to the basis $\cE$ for $\omega$. For the particular matrix $A(t)$ given by \eqref{eq.12.1} it is clear that every column of $A(t)$ belongs to $c_0$. Accordingly, 
	Theorem 3.9 of \cite{MRT} mentioned above yields the following result.
	
	\begin{prop}\label{P.12.1} Let $1\leq d_n\uparrow\infty$ and $t\in [0,1]$. Then $C_t$ exists in $h_d$, that is, $C_t\in \cL(h_d)$ if and only if the sequence
			\begin{equation}\label{eq.12.2}
				\left(\frac{1}{d_m}\sum_{n=0}^\infty d_n\left|\sum_{k=0}^m (a_{n,k}(t)-a_{n+1,k}(t))\right|\right)_{m\in\N_0}\in\ell^\infty.
				\end{equation}
		\end{prop}

The requirement \eqref{eq.12.2} is rather involved and difficult to verify for general $d$. For a large class of sequences $d$ we now show that the factorization of $C_t$, akin to \eqref{eq.10.2} and \eqref{eq.11.4}, provides a useful approach to study certain spectral properties of $C_t$.

The restriction to $h_d$ of the diagonal operator $D_\varphi\in \cL(\omega)$, as given by \eqref{Dia-op}, is again denoted by $D_\varphi$. The following result occurs in \cite[Example 4.12]{MRT}.

\begin{lemma}\label{L.12.2}Let the sequence $d$ satisfy $1\leq d_n\uparrow \infty$. Then $D_\varphi\in \cL(h_d)$ is a compact operator.
\end{lemma}

Let $1\leq d_n\uparrow \infty$. The elements of $\cE\subseteq h_d$ satisfy
\begin{equation}\label{eq.12.3}
	\|e_0\|_{h_d}=d_0\ \mbox{ and }\  \|e_n\|_{h_d}=d_{n-1}+d_n,\ n\geq 1.
	\end{equation}
For the right-shift operator $S\in \cL(\omega)$, as given by \eqref{10.1}, note that
\begin{equation}\label{eq.12.4}
	S^mx=(0,\ldots, 0, x_0,x_1,\ldots),\quad m\in\N_0,
	\end{equation}
for each $x\in\omega$, where $0$ occurs $m$-times prior to $x_0$. Fix $m\in\N_0$. Then, for each $x\in h_d$, we have
\begin{align*}
	\|S^{m+1}x\|_{h_d}&=d_m|x_0|+d_{m+1}|x_1-x_0|+d_{m+2}|x_2-x_1|+\ldots\\
	&=d_m|x_0|+\frac{d_{m+1}}{d_0}d_0|x_1-x_0|+\frac{d_{m+2}}{d_1}d_1|x_2-x_1|+\ldots\\
	&\leq d_m|x_0|+\left(\sup_{k\in\N_0}\frac{d_{m+k+1}}{d_k}\right)\sum_{k=0}^\infty d_k|x_{k+1}-x_k|\\
	&=d_m|x_0|+\|x\|_{h_d}.\sup_{k\in\N_0}\frac{d_{m+k+1}}{d_k}.
	\end{align*}
Since $\sum_{k=0}^\infty|x_{k+1}-x_k|\leq \sum_{k=0}^\infty d_k|x_{k+1}-x_k|=\|x\|_{h_d}$, the series $\sum_{k=0}^\infty(x_{k+1}-x_k)$ is absolutely convergent in $\C$. Moreover, $x\in c_0$ and so the identities $(x_{n+1}-x_0)=\sum_{k=0}^n(x_{k+1}-x_k)$, for each $n\in\N_0$, imply that $\sum_{k=0}^\infty (x_{k+1}-x_k)=-x_0$. Hence,
\[
d_m|x_0|\leq d_m\sum_{k=0}^\infty|x_{k+1}-x_k|\leq d_m\sum_{k=0}^\infty d_k|x_{k+1}-x_k|=d_m\|x\|_{h_d},
\]
because $d_k\geq 1$ for all $k\in\N_0$. It follows that 
\begin{equation}\label{eq.12.5}
	\|S^{m+1}\|_{h_d\to h_d}\leq d_m+\sup_{k\in\N_0}\frac{d_{m+k+1}}{d_k},\quad m\in\N_0.
	\end{equation}

\begin{lemma}\label{L.12.3}Let the sequence $1\leq d_n\uparrow\infty$ have the property that each sequence
	\begin{equation}\label{eq.12.6}
		\left(\frac{d_{m+k+1}}{d_k}\right)_{k\in\N_0},\quad m\in\N_0,
		\end{equation}
is decreasing.	Then 
	\begin{equation}\label{eq.12.7}
		\|S^{m+1}\|_{h_d\to h_d}\leq d_m+\frac{d_{m+1}}{d_0},\quad m\in\N_0.
	\end{equation}
\end{lemma}

\begin{proof} Fix $m\in\N_0$. Since the sequence \eqref{eq.12.6} is decreasing we have $\sup_{k\in\N_0}\frac{d_{m+k+1}}{d_k}=\frac{d_{m+1}}{d_0}$. Then \eqref{eq.12.7} follows from \eqref{eq.12.5}.
	\end{proof}

The following result provides a sufficient condition  to ensure the convergence of the series $\sum_{m=0}^\infty t^mS^m$ in $\cL_b(h_d)$ for each $t\in [0,1)$ and  a large class of sequences $d$.

\begin{prop}\label{P.12.4} Let  $1\leq d_n\uparrow\infty$ have the property that the sequence \eqref{eq.12.6} is decreasing for each $m\in\N_0$ and that $\frac{d_{k+1}}{d_k}\downarrow 1$. Then  the series $\sum_{m=0}^\infty t^mS^m$ is absolutely convergent in $\cL_b(h_d)$ for every $t\in [0,1)$.
	\end{prop}

\begin{proof} Note that $d_m+\frac{d_{m+1}}{d_0}\leq (1+\frac{1}{d_0})d_{m+1}\leq 2d_{m+1}$, for $m\in\N_0$. So, \eqref{eq.12.7} yields 
	\[
	\sum_{m=0}^\infty t^m\|S^m\|_{h_d\to h_d}\leq 1+2\sum_{m=1}^\infty d_mt^m<\infty,
	\]
	via the ratio test as $(d_{m+1}t^{m+1})/(d_mt^m)=t(d_{m+1}/d_m)\to t$ for $m\to\infty$. So, the series  $\sum_{m=0}^\infty t^mS^m$ converges absolutely  in $\cL(h_d)$ for the operator norm.
	\end{proof}

\begin{theorem}\label{T.12.5} Let  $1\leq d_n\uparrow\infty$ have the property that the sequence \eqref{eq.12.6} is decreasing for each $m\in\N_0$ and that $\frac{d_{k+1}}{d_k}\downarrow 1$. Then, for every $t\in [0,1)$, the operator $C_t\in \cL(h_d)$ is compact with spectra
	\begin{equation}\label{eq.12.8}
	\sigma_{pt}(C_t;h_d)=\Lambda\ \mbox{ and }\ \sigma(C_t;h_d)=\Lambda_0.
	\end{equation}
Furthemore, $C_t$ is power bounded and uniformly ergodic in $h_d$, but not supercyclic.
\end{theorem}

\begin{proof}
Fix $t\in [0,1)$. Proposition \ref{P.12.4} implies that $R_t:=\sum_{m=0}^\infty t^mS^m\in \cL(h_d)$ with the series absolutely convergent in $\cL_b(h_d)$. So, Lemma \ref{L.12.2} implies that 
\begin{equation}\label{eq.12.9}
	D_\varphi R_t=\sum_{m=0}^\infty t^m D_\varphi S^m\in \cL(h_d)
	\end{equation}
is a compact operator. Since $D_\varphi R_t=C_t$ holds in $\cL(\omega)$, \cite[Lemma 3.2]{CR4}, we see that $C^\omega_t(h_d)\subseteq h_d$ and hence, $C_t\in \cL(h_d)$ with $C_t$ compact via \eqref{eq.12.9}.

To establish \eqref{eq.12.8} it suffices to show that $\{x_t^{[m]}\, :\, m\in\N_0\}\subseteq h_d$; see Lemma \ref{L-2Sez2}. So, fix $m\in\N_0$. It follows from \eqref{eq.EigenvalueC} that
\begin{equation}\label{eq.12.10}
	(x_t^{[m]})_{m+n+1}=\frac{t(m+n+1)}{(n+1)}(x_t^{[m]})_{m+n},\quad n\geq 1,
	\end{equation}
and hence, that
\[
(x_t^{[m]})_{m+n+1}-(x_t^{[m]})_{m+n}=\frac{(mt+(n+1)(t-1))}{(n+1)}(x_t^{[m]})_{m+n},\quad n\geq 1.
\]
These identities imply that
\[
\beta_n:=d_{m+n}|(x_t^{[m]})_{m+n+1}-(x_t^{[m]})_{m+n}|=d_{m+n}\left|\frac{(mt+(n+1)(t-1))}{(n+1)}\right|(x_t^{[m]})_{m+n},
\]
for $n\geq 1$. Forming the ratios $\beta_{n+1}/\beta_n$ and using \eqref{eq.12.10} leads to
\begin{align*}
\frac{\beta_{n+1}}{\beta_n}&=\frac{d_{m+n+1}(n+1)}{d_{m+n}(n+2)}\left|\frac{mt+(n+2)(t-1)}{mt+(n+1)(t-1)}\right|\frac{(x_t^{[m]})_{m+n+1}}{(x_t^{[m]})_{m+n}}\\
&=\frac{d_{m+n+1}}{d_{m+n}}\frac{t(m+n+1)}{(n+2)}\left|\frac{mt+(n+2)(t-1)}{mt+(n+1)(t-1)}\right|,
\end{align*}
for each $n\geq 1$. Noting (with respect to $n$) that $\frac{d_{m+n+1}}{d_{m+n}}\downarrow 1$ implies  $\lim_{n\to\infty}\beta_{n+1}/\beta_n=t\in [0,1)$, the ratio test implies that  
\[
\sum_{n=1}^\infty d_{m+n}|(x_t^{[m]})_{m+n+1}-(x_t^{[m]})_{m+n}|<\infty
\]
which, in turn, implies that
\[
\|x_t^{[m]}\|_{h_d}=\sum_{k=0}^m d_k|(x_t^{[m]})_{k+1}-(x_t^{[m]})_{k}|+\sum_{n=1}^\infty d_{m+n}|(x_t^{[m]})_{m+n+1}-(x_t^{[m]})_{m+n}|<\infty.
\]
Since $m\in\N_0$ is arbitrary, we can conclude that \eqref{eq.12.8}  is valid.

According to Theorem \ref{T-general} the compact operator $C_t\in \cL(h_d)$ is power bounded and uniformly mean ergodic. Lemma \ref{L-nuovo} implies that $C_t$ is not supercyclic in $h_d$.\end{proof}

Let us consider some examples.

\begin{example}\label{Ex.12.6} \rm (i) Fix $r>0$ and let $d:=((n+1)^r)_{n\in\N_0}$, in which case $1\leq d_n\uparrow\infty$. Given $m\in\N_0$ the sequence \eqref{eq.12.6} is given by 
	\[
	\frac{d_{m+k+1}}{d_k}=\left(\frac{m+k+2}{k+1}\right)^r=\left(1+\left(\frac{m+1}{k+1}\right)\right)^r, \quad k\in\N_0, 
	\]
	which is decreasing. Moreover, for $m=0$, we see that $\frac{d_{k+1}}{d_k}=\left(1+\left(\frac{1}{k+1}\right)\right)^r\downarrow 1$. So, Theorem \ref{T.12.5} applies to $d$. For $r=1$, observe that $d=((n+1))_{n\in\N_0}$ and so $h_d$ is the classical  Hahn space $d$.
	
	Since the series in \eqref{eq.Nuova}, namely $\sum_{k=0}^\infty\frac{d_k}{(k+1)(k+2)}=\sum_{k=0}^\infty\frac{(k+1)^{r-1}}{(k+2)}$ is divergent, it follows that $C_1^\omega e_0\not\in h_d$. Accordingly, $C_1$ does not exist in $h_d$.
	
	(ii) Let $d:=(\log (n+3))_{n\in\N_0}$, in which case $1\leq d_n\uparrow\infty$. Fix $m\in\N_0$. Then  the sequence \eqref{eq.12.6} is given by 
	\[
	\frac{d_{m+k+1}}{d_k}=\frac{\log(m+k+4)}{\log(k+3)},\quad k\in\N_0.
	\]
	Define the function $f_m\colon (-1/2,\infty)$ by 
	\[
	f_m(x):=\frac{\log(x+m+4)}{\log(x+3)},\quad x>-\frac{1}{2},
	\]
	whose derivative is given by
	\begin{equation}\label{eq.12.11}
		f'_m(x)=\frac{(x+3)\log (x+3)-(x+m+4)\log (x+m+4)}{(x+3)(x+m+4)(\log (x+3))^2},\quad x>-\frac{1}{2}.
		\end{equation}
	For $s>1$ we observe that $(s\log (s))'=1+\log (s)>0$ and so $s\mapsto s\log (s)$ is an increasing function on $(0,\infty)$. This implies, via \eqref{eq.12.11}, that $f'_m(x)<0$ on $(-1/2,\infty)$ and hence, that $f_m$ is a decreasing function on $(-1/2,\infty)$. In particular, the sequence $\frac{d_{m+k+1}}{d_k}=f_m(k)$, for $k\in\N_0$, is decreasing. In addition, for $m=0$, it is clear that $f_0(k)=\frac{d_{k+1}}{d_k}=\frac{\log(k+4)}{\log(k+3)}\downarrow 1$. So, Theorem \ref{T.12.5} applies to $d$.
	\end{example}

	In contrast to Example \ref{Ex.12.6}(i), for $t=1$ and $d$ as in Example \ref{Ex.12.6}(ii) we have the following result. Recall that a  complex number $\lambda$ belongs to the \textit{residual spectrum} $\sigma_r(T;X)$ of a Banach space operator $T$ precisely when $\Ker(\lambda I-T)=\{0\}$ and ${\rm Im}(\lambda I-T)$ is not dense in $X$. 
	
	\begin{prop}\label{P.Nuova_12} Let $d=(\log (n+3))_{n\in\N_0}$, in which case $1\leq d_n\uparrow \infty$.
		\begin{itemize}
			\item[\rm (i)] The Cesàro operator $C_1$ given in \eqref{Ces-1} exists in $h_d$, that is, $C_1\in \cL(h_d)$.
			\item[\rm (ii)] The point spectrum $\sigma_{pt}(C_1;h_d)=\emptyset$.
			\item[\rm (iii)] The inclusions
			\begin{equation}\label{eq.In.12}
				\left\{z\in\C\,:\,\left|z-\frac{1}{2}\right|<\frac{1}{2}\right\}\cup\{1\}\subseteq \sigma_{pt}(C_1';h_d')\subseteq \sigma_r(C_1;h_d)
				\end{equation}
			are valid. In particular, 
			$$
			\left\{z\in\C\,:\,\left|z-\frac{1}{2}\right|\leq \frac{1}{2}\right\}\subseteq \sigma(C_1;h_d).
			$$
			\item[\rm (iv)] The operator $C_1\in \cL(h_d)$ is not compact, not mean ergodic and not supercyclic.
		\end{itemize}
	\end{prop}
	
	\begin{proof} (i)
	Consider the matrix $A(1)$ given by \eqref{eq.12.1} when $t=1$, in  which case the condition \eqref{eq.12.2} reduces to verifying that
	\begin{equation}\label{eq.es.1}
		\left(\frac{1}{d_m}\sum_{n=0}^\infty d_n\left|\sum_{k=0}^m(a_{n,k}(1)-a_{n+1,k}(1))\right|\right)_{m\in\N_0}\in\ell^\infty.
		\end{equation}
	Given $m,n\in\N_0$ direct calculation yields 
	\begin{equation*}
	\sum_{k=0}^m(a_{n,k}(1)-a_{n+1,k}(1))=\left\{\begin{array}{ll}\sum_{k=0}^m(\frac{1}{n+1}-\frac{1}{n+2})=\frac{(m+1)}{(n+1)(n+2)},& n\geq m\\
\left(\sum_{k=0}^m(\frac{1}{n+1}-\frac{1}{n+2})\right)-\frac{1}{n+2}=0, &0\leq n<m.	\end{array}\right.
	\end{equation*}
So, the coordinates of the series in \eqref{eq.es.1} are given, for each $m\in\N_0$, by
\begin{equation}\label{eq.es.2}
	\frac{1}{d_m}\sum_{n=0}^\infty d_n\left|\sum_{k=0}^m(a_{n,k}(1)-a_{n+1,k}(1))\right|=\frac{(m+1)}{\log (m+3)}\sum_{n=m}^\infty\frac{\log (n+3)}{(n+1)(n+2)}.
	\end{equation}
The function $g(x):=\frac{\log (x+3)}{(x+1)(x+2)}$, for $x\in [0,\infty)$, is positive, decreasing and integrable. By the integral test for series we can conclude that
\begin{equation}\label{eq.Integral Test}
\sum_{n=m}^\infty\frac{\log (n+3)}{(n+1)(n+2)}\leq \int_{m-1}^\infty g(x)\, dx,\quad m\geq 1.
\end{equation}
Define positive functions on $[1,\infty)$ by $F(y):=\int_{y-1}^\infty g(x)\, dx$ and $G(y):=\frac{\log (y+3)}{y+1}$, in which case $\lim_{y\to\infty}F(y)=0=\lim_{y\to\infty}G(y)$. The claim is that $\lim_{y\to\infty}\frac{F(y)}{G(y)}=1$, for which it suffices to show (via l'H\^{o}pital's rule) that $\lim_{y\to\infty}\frac{F'(y)}{G'(y)}=1$. Since $F'(y)=-g(y-1)=-\frac{\log (y+2)}{y(y+1)}$ and
\[
G'(y)=\frac{(y+1)-(y+3)\log (y+3)}{(y+3)(y+1)^2}\simeq -\frac{\log (y+3)}{(y+1)^2}, \ \mbox{for $y$ large},
\]
it follows that
\[
\frac{F'(y)}{G'(y)}\simeq -\frac{\log (y+2)}{y(y+1)}\cdot\left(-\frac{(y+1)^2}{\log (y+3)}\right)=\frac{(y+1)\log (y+2)}{y\log (y+3)}, \ \mbox{for $y$ large}.
\]
Accordingly, $\lim_{y\to\infty}\frac{F'(y)}{G'(y)}=1$ and hence, also
\[
\lim_{y\to\infty}\frac{F(y)}{G(y)}=\lim_{y\to\infty} \frac{(y+1)}{\log (y+3)}\int_{y-1}^\infty \frac{\log (x+3)}{(x+1)(x+2)}\,dx=1.
\]
It follows from \eqref{eq.es.2} and \eqref{eq.Integral Test} that the series in \eqref{eq.es.1} does indeed belong to $\ell^\infty$ which, via Proposition \ref{P.12.1}, implies that $C_1\in\cL(h_d)$.

(ii) It is known that $C_1^\omega$ (see \eqref{Ces-1}) belongs to $\cL(\omega)$, \cite[p.289]{ABRN}. Moreover, Proposition 4.4 (see also its proof) in \cite{ABRN} shows that $\sigma_{pt}(C_1^\omega;\omega)=\Lambda$ and that the $1$-dimensional eigenspace corresponding to the eigenvalue $\frac{1}{m+1}$, for each $m\in\N_0$, is spanned by
\begin{equation}\label{eq.Autovettore.1}
	x_1^{[m]}:=\left(0,0,\ldots, 0,1,\frac{(m+1)!}{m!1!},\frac{(m+2)!}{m!2!},\frac{(m+3)!}{m!3!},\ldots\right)\in\omega,
	\end{equation}
where $1$ is in position $m$. Clearly, $x_1^{[0]}=\mathbbm{1}\not\in c_0$ and so $x_1^{[0]}\not\in h_d$. Fix $m\geq 1$. Then 
\[
\frac{(m+n)!}{m!n!}\geq \frac{(1+n)!}{m!n!}=\frac{(1+n)}{m!}\geq \frac{2}{m!},\quad n\geq 1.
\] 
So, \eqref{eq.Autovettore.1} implies that $x_1^{[m]}\not\in c_0$, that is, $x_1^{[m]}\not\in h_d$ for all $m\in\N_0$.

The proof of Lemma \ref{L-2Sez2} also applies for $t=1$ and so we can conclude from Lemma \ref{L-2Sez2}(i) with $X:=h_d$ (as $C_1\in \cL(h_d)$) that $\sigma_{pt}(C_1;h_d)=\emptyset$.

(iii) It was noted in Section 5 that $C_1\in \cL(c_0)$ satisfies
 $$
\sigma(C_1;c_0)=\left\{z\in \C\,:\, \left|z-\frac{1}{2}\right|\leq \frac{1}{2}\right\}.
$$
The dual operator $C_1'\in\cL(\ell^1)$ is given by
\begin{equation}\label{eq.dual 1}
C_1'z=\left(\sum_{k=n}^\infty\frac{z_k}{k+1}\right)_{n\in\N_0},\quad z=(z_n)_{n\in\N_0}\in\ell^1=c_0',
\end{equation}
\cite{Le}, \cite{Re}, and satisfies
\begin{equation}\label{eq.spettro p dual}
 \sigma_{pt}(C'_1;\ell^1)=	\left\{z\in \C\,:\, \left|z-\frac{1}{2}\right|< \frac{1}{2}\right\}\cup\{1\},
	\end{equation}
\cite[Theorem 2.4]{Ak}. It was noted earlier  that $\ell^1\subseteq h_d'=bs_d$ for the natural inclusion. Since $\cE$ is a basis for both $\ell^1$ and $h_d$, it follows that the dual operator $C_1'\in \cL(h'_d)$ of $C_1\in \cL(h_d)$ is also given by \eqref{eq.dual 1}, for each $z\in h'_d$. Accordingly, $\sigma_{pt}(C_1';\ell^1)\subseteq \sigma_{pt}(C'_1;h_d')$ and so, via \eqref{eq.spettro p dual}, we have
\[
\left\{z\in \C\,:\, \left|z-\frac{1}{2}\right|< \frac{1}{2}\right\}\cup\{1\}\subseteq \sigma_{pt}(C_1';h_d').
\]
Recalling that
\[
\sigma_{pt}(C_1';h_d')\subseteq \sigma_{pt}(C_1;h_d)\cup \sigma_r(C_1;h_d),
\]
\cite[Proposition I.1.14]{Dow}, and $\sigma_{pt}(C_1;h_d)=\emptyset$, we can readily deduce 
\eqref{eq.In.12}.

Since $\sigma_r(C_1;h_d)\subseteq \sigma(C_1;h_d)$ and $\sigma(C_1;h_d)$ is a closed set, the proof is complete.

(iv) It is clear from \eqref{eq.In.12} that  $\sigma(C_1;h_d)$ is an uncountable set. This implies that $C_1\in \cL(h_d)$ cannot be compact, \cite[Theorem II.2.21(i)]{Dow}. 

Lemma \ref{L-nuovo} implies that $C_1\in \cL(h_d)$ is not supercyclic.

Assume that $C_1\in \cL(h_d)$ is mean ergodic. Then
\[
h_d=\Ker(I-C_1)\oplus \overline{{\rm Im}(I-C_1)},
\]
\cite[Theorem II.1.3]{K}, and hence, by part (ii), $h_d= \overline{{\rm Im}(I-C_1)}$. This contradicts the fact that $1\in \sigma_r(C_1;h_d)$; see \eqref{eq.In.12}. So, $C_1$ is not mean ergodic in $h_d$.
\end{proof}

The method of proof of Proposition \ref{P.Nuova_12} is applicable in a more general setting. Let $X\subseteq \omega$ be any BK-space such that $\cE\subseteq X$ and $\cE$ is  a \textit{basis}; in the terminology of \cite[Definition 4.2.13]{Wi}, this means the space $X$ has the property AK. Suppose, in addition, that $X\subseteq c_0$ continuously which, in the event that $C_1\in\cL(X)$, implies that $\sigma_{pt}(C_1;X)=\emptyset$; see the proof of  Proposition \ref{P.Nuova_12}(ii). Property AK of $X$ yields that the dual Banach space $X'$ is isomorphic to the $\beta$-dual of $X$ and so every $z\in X'$ is of the form
\[
\langle x,z\rangle=\sum_{n=0}^\infty x_nz_n,\quad x=(x_n)_{n\in\N_0}\in X,
\]
for the unique sequence $z_n:=z(e_n)$, for $n\in\N_0$, \cite[Theorem 7.2.9]{Wi}. 

Suppose that $\cE\subseteq X'$ and that $\cE$ is a basis in $X'$. Since $\ell^1\subseteq X'$, the argument of Proposition \ref{P.Nuova_12}(iii) shows that $C_1'\in \cL(X')$ is given by \eqref{eq.dual 1}, for every $z\in X'$, and hence, $\left\{z\in \C\, :\, \left|z-\frac{1}{2}\right|<\frac{1}{2}\right\}\cup\{1\}\subseteq \sigma_{pt}(C_1';X')$. Since $\sigma_{pt}(C_1; X)=\emptyset$ it follows that  $\left\{z\in \C\, :\, \left|z-\frac{1}{2}\right|<\frac{1}{2}\right\}\cup\{1\}\subseteq \sigma_{r}(C_1;X)$. The proof of Proposition \ref{P.Nuova_12}(iv) carries over adverbatum. So, we have established the following result.


\begin{corollary}\label{C.Nuovo_12} Let $X\subseteq \omega$ be a BK-space such that $\cE\subseteq X$ and $\cE\subseteq X'$ and that $\cE$ is a basis for both $X$ and $X'$. Suppose that $X\subseteq c_0$ continuously and  $C_1\in \cL(X)$. Then 
	\[
	\sigma_{pt}(C_1;X)=\emptyset\  \mbox{and}\  \left\{z\in \C\, :\, \left|z-\frac{1}{2}\right|<\frac{1}{2}\right\}\cup\{1\}\subseteq \sigma_r(C_1;X).
	\]
	 Moreover, $C_1$ is not compact, not mean ergodic  and not supercyclic in $X$.
\end{corollary}

We now exhibit a sequence $1\leq d_n\uparrow \infty$ such that $C_t$ exists in $h_d$ for some $t\in [0,1)$ but, not for all $t$, and other sequences $1\leq d_n\uparrow \infty$ such that $C_t$ fails to exist in $h_d$ for every $t\in (0,1)$. Some preliminaries are needed.

It is immediate from \eqref{Ces-op} that
\begin{equation}\label{eq.12.13}
	C_t^\omega e_0=\left(\frac{t^n}{n+1}\right)_{n\in\N_0},\quad t\in (0,1).
	\end{equation}

\begin{lemma}\label{L.12.7} Let  $1\leq d_n\uparrow \infty$. For every $t\in (0,1)$ it is the case that
	\begin{equation}\label{eq.12.14}
		\sum_{k=0}^\infty d_k\left|\frac{t^k}{k+1}-\frac{t^{k+1}}{k+2}\right|\geq 	\sum_{k=0}^\infty \frac{d_kt^k}{(k+1)(k+2)}.
		\end{equation}
	\begin{itemize}
		\item[\rm (i)] Let $t\in (0,1)$ satisfy $\left( \frac{d_nt^n}{(n+1)(n+2)}\right)_{n\in\N_0}\not\in \ell^1$. Then $C_t$ does not exist in $h_d$.
		\item[\rm (ii)] Let $t\in (0,1)$ satisfy $\left( \frac{d_nt^n}{(n+1)(n+2)}\right)_{n\in\N_0}\not\in c_0$. Then $C_t$ does not exist in $h_d$.
	\end{itemize}
\end{lemma}

\begin{proof} Fix $t\in (0,1)$. Then 
	\begin{align*}
		&	\sum_{k=0}^\infty d_k\left|\frac{t^k}{k+1}-\frac{t^{k+1}}{k+2}\right|=\sum_{k=0}^\infty d_k t^k\left(\frac{1}{k+1}-\frac{t}{k+2}\right)\\
	&	\geq \sum_{k=0}^\infty d_k t^k\left(\frac{1}{k+1}-\frac{1}{k+2}\right)=	\sum_{k=0}^\infty \frac{d_kt^k}{(k+1)(k+2)},
	\end{align*}
	which is precisely \eqref{eq.12.14}.
	
	(i) For the given $t$ set $x:=C_t^\omega e_0$. Then \eqref{eq.12.14} implies that
\[
	\sum_{k=0}^\infty d_k |x_{k+1}-x_k|\geq \left\|\left( \frac{d_nt^n}{(n+1)(n+2)}\right)_{n\in\N_0}\right\|_1=\infty,
	\]
	and hence, via \eqref{eq.12.13}, we can conclude that $C_t^\omega e_0\not\in h_d$. So, $C_t^\omega (h_d)\not \subseteq h_d$, that is, $C_t$ does not exist in $h_d$.
	
	(ii) Since $\ell^1\subseteq c_0$, this follows from part (i).
\end{proof}

	 \begin{example}\label{E.12.8}\rm
	(i) For $\alpha>1$ define $d:=(\alpha^n)_{n\in\N_0}$, in which case $1\leq d_n\uparrow \infty$. Fix $m\in\N_0$. Then the sequence $\frac{d_{m+k+1}}{d_k}=\alpha^m$, for $k\in\N_0$, is decreasing (not strictly) but, $\frac{d_{k+1}}{d_k}=\alpha$ for $k\in\N_0$ does \textit{not} decrease to $1$. So, Theorem \ref{T.12.5} is not applicable.
	
	Since $\sup_{k\in\N_0}\frac{d_{m+k+1}}{d_k}=\alpha^m$, it follows from \eqref{eq.12.5} that $\|S^{m+1}\|_{h_d\to h_d}\leq 2\alpha^m$, for $m\geq 1$, which implies that
	\[
	\sum_{m=0}^\infty t^m\|S^{m+1}\|_{h_d\to h_d}\leq 1+2\sum_{m=1}^\infty t^m\alpha^m<\infty
	\]
	\textit{provided} that $t\in [0,1/\alpha)$. For such $t$ the series $R_t:=\sum_{m=0}^\infty t^mS^m$ converges absolutely in $\cL_b(h_d)$.
	
	Given $n\in\N_0$, it follows from \eqref{eq.12.4} that
	\[
	\sum_{m=0}^\infty t^mS^me_n=(0,\ldots, 0, 1,t, t^2,\ldots)\in\omega
	\]
	for each $t\in [0,1)$, where $1$ is in position $n$, and so
	\begin{align*}
		\|\sum_{m=0}^\infty t^mS^me_n\|_{h_d}&=d_n|1-0|+d_{n+1}|t-1|+d_{n+2}|t^2-t|+\ldots\\
			&=\alpha^n+(1-t)\sum_{k=1}^\infty \alpha^{n+k}t^{k-1}\\
			&=\alpha^n+(1-t)\alpha^{n+1}\sum_{k=0}^\infty (\alpha t)^{k}<\infty
		\end{align*}
if and only if $t\in [0,1/\alpha)$. Accordingly, a necessary condition for the series $R_t=\sum_{m=0}^\infty t^mS^m$ to converge in $\cL_s(h_d)$ is that  $t\in [0,1/\alpha)$. As observed above, for such $t$ the series for $R_t$ actually converges absolutely for the operator norm in $\cL(h_d)$. Arguing as in first paragraph of the proof of Theorem \ref{T.12.5} it follows that $C_t\in \cL(h_d)$, for every $t\in [0,1/\alpha)$, is a compact operator.

Note that $\frac{d_{m+n+1}}{d_{m+n}}=\alpha$ for all $m,n\in\N_0$. For $m\in\N_0$ fixed, the sequence $(\beta_n)_{n=1}^\infty$  as defined  in the proof of Theorem \ref{T.12.5} (now with $d_k=\alpha^k$, for $k\in\N_0$) satisfies
\[
\lim_{n\to\infty}\frac{\beta_{n+1}}{\beta_n}=\lim_{n\to\infty}\frac{\alpha t(m+n+1)}{(n+2)}\left|\frac{mt+(n+2)(t-1)}{mt+(n+1)(t-1)}\right|=\alpha t.
\]
The argument (via the ratio test for $(\beta_n)_{n=1}^\infty$) used in the proof of Theorem \ref{T.12.5} can be repeated to show that $\{x_t^{[m]}\, :\, m\in\N_0\}\subseteq h_d$ whenever  $t\in [0,1/\alpha)$ and hence, via Lemma \ref{L-2Sez2}, that
\[
\sigma_{pt}(C_t; h_d)=\Lambda\ \mbox{ and }\ \sigma(C_t;h_d)=\Lambda_0,\quad t\in [0,1/\alpha).
\]
Finally, for every $t\in [0,1/\alpha)$, Theorem \ref{T-general} and Lemma \ref{L-nuovo} yield that $C_t\in \cL(h_d)$ is power bounded and uniformly mean ergodic but, not supercyclic.

For the remaining $t\in [1/\alpha, 1)$ the claim is that $C_t$ does \textit{not} exist in $h_d$. First consider $t\in (1/\alpha, 1)$. Then $\alpha t>1$ and so 
\[
\left(\frac{d_nt^n}{(n+1)(n+2)}\right)_{n\in\N_0}=\left(\frac{(\alpha t)^n}{(n+1)(n+2)}\right)_{n\in\N_0}\not\in c_0.
\]
Lemma \ref{L.12.7}(ii) shows that $C_t$ does not exist in $h_d$. 

For the case $t=1/\alpha$ (i.e., $\alpha t=1$) observe that $x:=C_t^\omega e_0$ satisfies
\begin{align*}
	&\sum_{k=0}^\infty d_k |x_{k+1}-x_k|=\sum_{k=0}^\infty \alpha^k \left|\frac{t^{k+1}}{k+2}-\frac{t^k}{k+1}\right|=\sum_{k=0}^\infty (\alpha t)^k\left|\frac{t}{k+2}-\frac{1}{k+1}\right|\\
	&= \sum_{k=0}^\infty \left(\frac{1}{k+1}-\frac{t}{k+2}\right)=\sum_{k=0}^\infty\frac{(1-t)k+(2-t)}{(k+1)(k+2)}\geq \sum_{k=0}^\infty\frac{(1-t)k+(1-t)}{(k+1)(k+2)}\\
	&= (1-t)\sum_{k=0}^\infty\frac{1}{k+2}=\infty.
	\end{align*}
Hence, \eqref{eq.12.13} implies that $C_t^\omega e_0\not\in h_d$, that is, $C_{1/\alpha}$ does not exist in $h_d$.

Concerning $t=1$, observe that \eqref{eq.Nuova} is the divergent series $\sum_{k=0}^\infty \frac{d_k}{(k+1)(k+2)}=\sum_{k=0}^\infty \frac{\alpha^k}{(k+1)(k+2)}$. Hence, $C_1^\omega e_0\not\in h_d$ which implies that $C_1$ does not exist in $h_d$.

(ii)
 Let $d:=((n+1)!)_{n\in\N_0}$, in which case $1\leq d_n\uparrow \infty$. Fix $t\in (0,1)$. Then the sequence
 \[
\gamma_n:=\frac{d_{n}t^n}{(n+1)(n+2)}=\frac{(n+1)!t^n}{(n+1)(n+2)}=\frac{n!t^n}{n+2},\quad n\in\N_0,
\]
satisfies $\gamma_n\geq \frac{t}{2}(n-1)! t^{n-1}$, for all $n\geq 1$, from which it is clear that $(\gamma_n)_{n\in\N_0}\not\in c_0$. So, Lemma \ref{L.12.7}(ii) implies that $C_t$ fails to exist in $h_d$ for every $t\in (0,1)$. For $t=0$ we recall that $C_0=D_\varphi\in \cL(h_d)$ exists and is compact; see Lemma \ref{L.12.2}. Arguing as in part (i) it follows that $C_1$ does not exist in $h_d$.

(iii) Define $d'=((n+1)^{n+1})_{n\in\N_0}$, in which case $1\leq d_n'\uparrow \infty$. Since $d_n=(n+1)!\leq (n+1)^{n+1}$, for all $n\in\N_0$, with $(d_n)_{n\in\N_0}$ as in part (ii), it follows from part (ii) that $\gamma_n\leq \frac{d_n't^n}{(n+1)(n+2)}$ for $n\in\N_0$ and hence, $\left(\frac{d_n't^n}{(n+1)(n+2)}\right)_{n\in\N_0}\not \in c_0$. So, $C_t$ fails to exist in $h_{d'}$ for every $t\in (0,1)$. Arguing  as in part (i) it follows that $C_1$ does not exist in $h_{d'}$ either.
	\end{example}

\textbf{Acknowledgements.} The authors thank the referee for the detailed report which included many useful comments and suggestions which improved the presentation of the paper. 

 The research of J. Bonet was partially supported by the project 
PID2020-119457GB-100 funded by MCIN/AEI/10.13039/501100011033 and by 
``ERFD A way of making Europe''.

\medskip

\textbf{Data Availability Statement.} The manuscript has no associated data.

\medskip

\textbf{Compliance with Ethical Standards}

\medskip
\textbf{Conflict of interest.} The authors declare that they have no conflicts of interest regarding this study.

\bibliographystyle{plain}

\end{document}